\documentclass{article}
\usepackage[utf8]{inputenc}
\usepackage{amsmath,amssymb,amsfonts,bm}
\usepackage{graphicx}
\usepackage{epstopdf}
\usepackage{mathrsfs}
\usepackage{amsthm}
\usepackage{galois,booktabs,ctable,threeparttable,tabularx,algorithm,algorithmic}
\usepackage[misc]{ifsym}
\usepackage{authblk}
\usepackage{subfigure}
\usepackage{geometry,fullpage,cite}

\title{Sparse Signal Recovery from Phaseless Measurements via Hard Thresholding Pursuit}
\author[a]{Jian-Feng Cai}
\author[b]{Jingzhi Li}
\author[c]{Xiliang Lu}
\author[a]{Juntao You\thanks{Corresponding author: jyouab@connect.ust.hk}}
\affil[a]{Department of Mathematics\\ The Hong Kong University of Science and Technology\\ Clear Water Bay, Kowloon, Hong Kong SAR, China.}
\affil[b]{Department of Mathematics, International Center of Mathematics, and Guangdong Provincial
Key Laboratory for Computational Science and Material Design\\ Southern University of Science and Technology\\ Shenzhen 518005, China.}
\affil[c]{School of Mathematics and Statistics, and Hubei Key Laboratory of Computational Science\\ Wuhan University\\ Wuhan 430072, China.}

\def\lV{\left\lVert}
\def\rV{\right\lVert}
\def\lv{\left\lvert}
\def\rv{\right\lvert}

\def\l{\langle}
\def\r{\rangle}

\def\H{\mathcal H}

\def\G{\mathcal G}
\def\I{\mathcal I}

\def\sS{\mathcal S}

\def\T{\mathcal T}

\newtheorem{theorem}{Theorem}
\newtheorem{lemma}{Lemma}

\newtheorem{corollary}{Corollary}
\newtheorem{proposition}{Proposition}

\graphicspath{{figures/}}

\begin{document}

\maketitle
\begin{abstract}
In this paper, we consider the sparse phase retrieval problem, recovering an $s$-sparse signal $\bm{x}^{\natural}\in\mathbb{R}^n$ from $m$ phaseless samples $y_i=|\langle\bm{x}^{\natural},\bm{a}_i\rangle|$ for $i=1,\ldots,m$. Existing sparse phase retrieval algorithms are usually first-order and hence converge at most linearly. Inspired by the hard thresholding pursuit (HTP) algorithm in compressed sensing, we propose an efficient second-order algorithm for sparse phase retrieval. Our proposed algorithm is theoretically guaranteed to give an exact sparse signal recovery in finite (in particular, at most $O(\log m + \log(\|\bm{x}^{\natural}\|_2/|x_{\min}^{\natural}|))$ steps, when $\{\bm{a}_i\}_{i=1}^{m}$ are i.i.d. standard Gaussian random vector with $m\sim O(s\log(n/s))$ and the initialization is in a neighbourhood of the underlying sparse signal. Together with a spectral initialization, our algorithm is guaranteed to have an exact recovery from $O(s^2\log n)$ samples. Since the computational cost per iteration of our proposed algorithm is the same order as popular first-order algorithms, our algorithm is extremely efficient. Experimental results show that our algorithm can be several times faster than existing sparse phase retrieval algorithms.
\end{abstract}

\section{Introduction}
\subsection{Phase retrieval problem}
The phase retrieval problem is to recover an $n$-dimensional signal $\bm{x}^{\natural}$ from a system of phaseless equations
\begin{align}\label{problem:PR}
y_i=| \langle \bm{a}_i,\bm{x}^{\natural} \rangle | , \quad i=1,2,\cdots,m,
\end{align}
where $\bm{x}^{\natural}$ is the unknown vector to be recovered, $\bm{a}_i$ for $i=1,\ldots,m$ are given sensing vectors,  $y_i\in\mathbb{R}_+$ for $i=1,\ldots,m$ are observed modulus data, and $m$ is the number of measurements (or the sample size). This problem arises in many fields such as X-ray crystallography \cite{harrison1993phase}, optics \cite{walther1963question}, microscopy \cite{miao2008extending}, and others \cite{fienup1982phase}.

The classical approaches for phase retrieval were mostly based on alternating projections, e.g., the work of Gerchberg and Saxton \cite{gerchberg1972practical} and Fienup \cite{fienup1982phase}, which usually work very well empirically but lack provable guarantees in the primary literatures. Recently, lots of attentions has been paid to constructing efficient algorithms with theoretical guarantees when given certain classes of sampling vectors. In these approaches, one of the main targets is to achieve optimal sampling complexity for, e.g, reducing the cost of sampling and computation. They are categorized into convex and nonconvex optimization based approaches. Typical convex approaches such as phaselift \cite{candes2015phase} transfer the phase retrieval problem into a semi-definite programming(SDP), which lifts the unknown $n$-dimensional signal to an $n\times n$ matrix and thus computationally expensive. To overcome this, some other convex approaches such as Phasemax \cite{goldstein2016phasemax} and others \cite{hand2018elementary,bahmani2017flexible} solve convex optimizations with $n$ unknowns only. However, these convex formulations depend highly on the so-called anchor vectors that approximate the unknown signal, and the sampling complexity might be unnecessarily large if the anchor vector is not good enough. Meanwhile, nonconvex optimization based approaches were proposed and studied in the past years. Examples include alternating minimization \cite{netrapalli2013phase} (or Fienup methods), Wirtinger flow \cite{candes2015phase1},  Kaczmarz \cite{tan2019phase,wei2015solving}, Riemannian optimization \cite{cai2018solving}, and Gauss-Newton \cite{gao2016gauss,ma2018globally}. To prove the guarantee, these algorithms normally require a good initialization close enough to the ground truth, which is achieved by spectral initializations. Nevertheless, experimental results suggest that the designed initialization is not necessary --- a random initialization usually lead to the correct phase retrieval. To explain this, either the algorithms with random initialization is studied \cite{waldspurger2018phase,chen2019gradient,tan2019online}, or the global geometric landscape of nonconvex objective functions are examined \cite{8918236,sun2018geometric}, showing that there is actually no spurious local minima.

\subsection{Sparse phase retrieval}
All the aforementioned provable algorithms need a sampling complexity $m\sim O(n\log^a n)$ with $a\geq 0$. This sampling complexity is (nearly) optimal, since the phase retrieval problem is solvable only when $m\geq 2n-1$ and $m\geq 4n-4$ for real and complex signals respectively \cite{balan2006signal,conca2015algebraic}. Nevertheless, there is still a demand to further reduce the sampling complexity to save the cost of sampling. We have to exploit the structure of the underlying signal. In many applications, especially those related to signal processing and imaging, the true signal $\bm{x}^{\natural}$ is known to be sparse or approximately sparse in a transform domain \cite{mallat1999wavelet}. Have this priori knowledge in mind, it is possible to recover the signal using only a small number (possibly sublinear in $n$) of phaseless samples.

For simplicity, we assume that $\bm{x}^{\natural}$ is sparse with sparsity at most $s$, i.e., $\|\bm{x}^{\natural}\|_0\leq s$, where $\|\cdot\|_0$ stands for the number of nonzero entries. With this sparsity constraint, the phase retrieval problem \eqref{problem:PR} can be reformulated as: find $\bm{x}^{\natural}$ such that
\begin{align}\label{problem:sPRori}
y_i=| \langle \bm{a}_i,\bm{x}^{\natural} \rangle |,  \quad i=1,2,\cdots,m, \qquad \text{subject to} \ \lVert \bm{x}^{\natural}\lVert_0\le s.
\end{align}
The problem \eqref{problem:sPRori} is referred to a sparse phase retrieval problem. It has been proved that the sample size $m=2s$ is necessary and sufficient to determine a unique solution for the problem \eqref{problem:sPRori} with generic  measurements in the real case \cite{wang2014phase}. Thus it opens the possibility for successful sparse phase retrieval using very few samples.

Though the sparse phase retrieval problem \eqref{problem:sPRori} is NP-hard in general, there are many available algorithms that are guaranteed to find $\bm{x}^{\natural}$ with overwhelming probability under certain class of random measurements. Examples of such algorithms are $\ell_{1}$-regularized PhaseLift method \cite{li2013sparse}, sparse AltMin \cite{netrapalli2013phase}, thresholding/projected Wirtinger flow \cite{cai2016optimal,soltanolkotabi2019structured}, SPARTA \cite{wang2016sparse}, CoPRAM \cite{jagatap2019sample}, and a two-stage strategy introduced in \cite{iwen2017robust}. All these approaches except for \cite{iwen2017robust}
\footnote{A two-stage sampling scheme is proposed in \cite{iwen2017robust}, where the first stage for sparse compressed sensing and the second stage for phase retrieval. It needs $O(s\log(n/s))$ samples, but the sampling scheme is complicated.}
are analyzed under Gaussian random measurements, showing $O(s^{2}\mathrm{log}\, n)$ random Gaussian measurements are sufficient to achieve a successful sparse phase retrieval. Though not optimal, this sampling complexity is much smaller than that in the general phase retrieval. For convex approaches, it has been shown in \cite{li2013sparse} that $O(s^2\log n)$ random Gaussian samples is necessary. For nonconvex approaches, the algorithms are ususally divided into two stages, namely the initialization stage and the refinement stage. In the initialization stage, a spectral initialization is performed, and it requires $O(s^2\log n)$ Gaussian random samples to achieve an estimation sufficiently close to the ground truth. In the refinement stage, the initial estimation is refined by different algorithms, most of which are able to converge to the ground truth linearly using $O(s\log (n/s))$ Gaussian random samples. Thus the sample complexity in total is dominated by the initialization stage.

\subsection{Our Contributions}
In this paper, we propose a simple yet efficient nonconvex algorithm with provable recovery guarantees for sparse phase retrieval. Similar to most of the existing nonconvex algorithms, our proposed algorithm is divided into two stages, and the initialization stage rely on a spectral initialization also. So, we cannot reduce the sampling complexity to optimal as well. Instead, we focus on the improvement of the computational efficiency in the refinement stage, using $O(s\log(n/s))$ random Gaussian samples. Different to existing algorithms that usually converges linearly, our proposed algorithm is proven to have the exact recovery of the sparse signal in at most $O(\log m + \log(\|\bm{x}^{\natural}\|_2/|x_{\min}^{\natural}|))$ steps, while it has almost the same computational cost per step as others. Therefore, our algorithm is much more efficient than existing algorithms. Experimental results confirm this, showing that our algorithm gives a very accurate recovery in very few iterations, and it gains several times acceleration over existing algorithms.

Our proposed algorithm is based on the hard thresholding pursuit (HTP) for compressed sensing introduced in \cite{foucart2011hard}. Building on the projected gradient descent (or iterative hard thresholding (IHT) \cite{blumensath2008iterative}), the idea of HTP is to project the current guess into the space that best match the measurements in each iteration. With the help of the restricted isometry property (RIP) \cite{candes2008restricted}, HTP for compressed sensing is proved to have a robust sparse signal recovery  in finite steps starting from any initial guess. Our proposed algorithm is an adoption of HTP from linear measurements to phaseless measurements. Giving a current iteration, we first estimate the phase of the phaseless measurements, and then one step of HTP iteration is applied. Our algorithm has almost the same computational cost as compressed sensing HTP, while preserving the convergence in finite steps (in particular, in $O(\log m + \log(\|\bm{x}^{\natural}\|_2/|x_{\min}^{\natural}|))$ steps). Since the HTP algorithm is a second order Newton's method (see e.g. \cite{huang2021unified,zhou2019global}), our algorithm can be viewed as a Newton's method for sparse phase retrieval.

\subsection{Notations and outline}
 For any $\bm{a}, \bm{b}\in \mathbb{R}^{n}$, we denote $\bm{a}\odot \bm{b}$ the entrywise product of $\bm{a}$ and $\bm{b}$, i.e., $\bm{a}\odot \bm{b}=[a_1b_1,~a_2b_2,~\cdots,~a_nb_n]^T$. For $\bm{x}\in \mathbb{R}^{n}$, $\mathrm{sgn}{\left(\bm{x}\right)}\in\mathbb{R}^{n}$ is defined by $\left[\mathrm{sgn}{\left(\bm{x}\right)}\right]_i=1$ if $x_i>0$, $\left[\mathrm{sgn}{\left(\bm{x}\right)}\right]_i=0$ if $x_i=0$, and
 $\left[\mathrm{sgn}{\left(\bm{x}\right)}\right]_i=-1$ if $x_i<0$. $\lVert \bm{x}\lVert_0$ is the number of nonzero entries of $\bm{x}\in\mathbb{R}^n$, and $\lVert \bm{x}\lVert_2$ is the standard $2$-norm, i.e. $\lVert \bm{x}\lVert_2=\left(\sum_{i=i}^{n}x_i^2\right)^{1/2}$. For a matrix $\bm{A}\in \mathbb{R}^{m\times n}$, $\bm{A}^{T}$ is its transpose, and $\lV\bm{A}\rV_2$ denotes its spectral norm. For an index set $\sS\subseteq\left\{1,2,\cdots,n\right\}$, $\bm{A}_{\sS}$ (or $[\bm{A}]_{\sS}$ sometimes) stands for the submatrix of a matrix $\bm{A}\in\mathbb{R}^{m\times n}$ obtained by keeping only the columns indexed by $\sS$, and $\bm{x}_{\sS}$ (or $[\bm{x}]_{\sS}$ sometimes) denotes the vector obtained by keeping only the components of a vector $\bm{x}\in\mathbb{R}^n$ indexed by $\sS$. The hard thresholding operator $\H_s~:~\mathbb{R}^{n}\to\mathbb{R}^n$ keeps the largest $s$ components in magnitude of a vector in $\mathbb{R}^n$ and sets the other ones to zero. For $a\in \mathbb{R}_{+}$, $\log a$ in this paper represents the logarithm of $a$ to the base $e$. For $\bm{x},\bm{x}^{\natural}\in \mathbb{R}^n$, the distance between $\bm{x}$ and $\bm{x}^{\natural}$ is defined as
\begin{equation}\label{eq:dist}
\mathrm{dist}\left(\bm{x},\bm{x}^{\natural}\right)=\min\left\{\lV\bm{x}-\bm{x}^{\natural} \rV_2,\lV\bm{x}+\bm{x}^{\natural} \rV_2 \right\}.
\end{equation}
Also, in the paper, $x_{\min}^{\natural}$ means the smallest nonzero entry in magnitude of $\bm{x}^{\natural}$.

The rest of papers are organized as follows. In Section~\ref{section:algorithm} and \ref{section:theory}, we introduce the details of the proposed algorithm and our main theoretical results, respectively. Numerical experiments to illustrate the performance of the algorithm are given in Section \ref{section:experiments}. The proofs are given in Section \ref{section:proofs}, and we conclude the paper in Section \ref{section:conclusion}.

\section{Algorithms}\label{section:algorithm}
In this section, we describe our proposed algorithm in detail. Similar to most of the existing non-convex sparse phase retrieval algorithms, our algorithms consists of two stages, namely, the initialization stage and the iterative refinement stage. Since the initialization stage can be done by an off-the-shelf algorithm such as spectral initialization, we will focus on the iterative refinement stage. We first give some related algorithms, especially iterative hard thresholding algorithms, in Section \ref{subsection:iht}, and then our proposed algorithm is presented in Section \ref{subsection:htp}.

To simplify the notations, we denote the sampling matrix and the observations by
\begin{equation}\label{Aandy}
\bm{A}:=\frac{1}{\sqrt{m}} [\bm{a}_1~\bm{a}_2~\ldots~\bm{a}_m]^T \in \mathbb{R} ^{m\times n},\qquad
\bm{y} := \frac{1}{\sqrt{m}}[y_1~y_2~\ldots~y_m]^T
\end{equation}
respectively, where $\bm{a}_i\in\mathbb{R}^n$ and $y_i\in\mathbb{R}_+$ for $i=1,\ldots,m$ are from \eqref{problem:PR} and \eqref{problem:sPRori}. Thus, the sparse phase retrieval problem \eqref{problem:sPRori} can be rewritten as to find $\bm{x}^{\natural}$ satisfying
 \begin{align}\label{problem:sPR}
 \bm{y}=\lvert \bm{A}\bm{x}^{\natural} \lvert, \quad \text{subject to} \ \lVert \bm{x}^{\natural}\lVert_0 \le s.
 \end{align}
There are several possible ways to solve \eqref{problem:sPR} by reformulating it into constrained minimizations with different objective functions.

\subsection{Iterative Hard Thresholding Algorithms}\label{subsection:iht}

One natural way to solve \eqref{problem:sPR} is to consider a straightforward least squares fitting to the amplitude equations in \eqref{problem:sPR} subject to the sparsity constraint, and we solve
\begin{align}\label{problem:leastsq}
\mathop{\mathrm{minimize}}\limits_{\lV\bm{x}\rV_0\le s}~f(\bm{x}),\quad\mbox{where}\quad f(\bm{x})=\frac{1}{2}\big\|\bm{y}-|\bm{A}\bm{x}| \big\|_{2}^{2}.
\end{align}
Though the objective function $f$ is non-smooth, a generalized gradient is available (see e.g., \cite{zhang2017nonconvex}), which is given as
$$\nabla f\left(\bm{x}\right)=\bm{A}^T\left(\bm{A}\bm{x}-\bm{y}\odot \mathrm{sgn}\left(\bm{A}\bm{x}\right)\right).$$
Furthermore, the projection onto the feasible set $\{\bm{x}:\|\bm{x}\|_0\leq s\}$ can be done efficiently by $\mathcal{H}_s$, though the feasible set is non-convex. Altogether, one may apply a projected gradient descent to solve \eqref{problem:leastsq}, yielding
\begin{align}\label{iter:iht}
\bm{x}_{k+1}=\H_s\Big(\bm{x}_k+\mu\bm{A}^T\left(\bm{y}\odot \mathrm{sgn}\left(\bm{A}\bm{x}_k\right)-\bm{A}\bm{x}_k\right)\Big).
\end{align}
This algorithm is an iterative hard thresholding (IHT) algorithm, since the hard thresholding operator $\mathcal{H}_s$ is applied in each iteration. This algorithm is analyzed in \cite{soltanolkotabi2019structured} under a general framework, which proved that \eqref{iter:iht} converges linearly to $\pm\bm{x}^{\natural}$ with high probability when it is initialized in a small neighbourhood of $\pm\bm{x}^{\natural}$ and $O(s\log(n/s))$ Gaussian random measurements vectors are used. By ruling out some outlier phaseless equations from the least squares fitting at each iteration according to some truncation rule, the algorithm \eqref{iter:iht} becomes the SPARTA algorithm \cite{wang2016sparse} with the sampling complexity $O(s\log (n/s))$ for sparse phase retrieval provided a good initialization. Without $\mathcal{H}_s$, the algorithm \eqref{iter:iht} and its variants are algorithms for the standard phase retrieval \eqref{problem:PR}, including the truncated amplitude flow (TAF) algorithm \cite{wang2017solving}, the reshaped Wirtinger flow (RWF) \cite{zhang2017nonconvex}, both of which are guaranteed to have an exact phase retrieval starting from a good initialization when $O(n)$ Gaussian random measurements are used.

Alternatively, one may square both sides of the equation in \eqref{problem:sPR} to obtain $\bm{y}^2=|\bm{A}\bm{x}^{\natural}|^2$, where the square of a vector is componentwise. The resulting equation is known as the intensity equation in phase retrieval. Then, we can solve \eqref{problem:sPR} by minimizing the least square error of the intensity equation subject to the sparsity constraint, which leads to solving the following constrained minimization
\begin{equation}\label{eq:conslsint}
\mathop{\mathrm{minimize}}\limits_{\lV\bm{x}\rV_0\le s}~f_I(\bm{x}),\quad\mbox{where}\quad f_I(\bm{x})=\frac{1}{2}\big\|\bm{y}^2-|\bm{A}\bm{x}|^2\big\|_{2}^{2}.
\end{equation}
The advantage of fitting the intensity equation is that the objective function $f_I$ is both smooth and local strongly convex (see \cite{ma2018implicit}). Together with the fact that the projection onto the $s$-sparse set can be easily done by $\mathcal{H}_s$, one can apply the projected gradient descent to solve \eqref{eq:conslsint} and obtain
\begin{equation}\label{eq:pwf}
\bm{x}_{k+1}=\H_s\left(\bm{x}_{k}-\mu\nabla f_I\left(\bm{x}_k\right)\right).
\end{equation}
This is again an iterative hard thresholding algorithm. Since the gradient should be taken in the Wirtinger derivative in the complex signal case, the algorithm \eqref{eq:pwf} is more widely known as projected Wirtinger flow \cite{cai2016optimal,soltanolkotabi2019structured}. It is proved that \eqref{eq:pwf} gives an exact sparse phase retrieval in a neighbourhood of $\pm\bm{x}^{\natural}$ when sampled by $O(s\log(n/s))$ random Gaussian measurements. When there is no $\mathcal{H}_s$, \eqref{eq:pwf} and a truncation variant are consistent with Wirtinger flow \cite{candes2015phase1} and truncated Wirtinger flow \cite{chen2015solving} algorithms respectively for the standard phase retrieval.

\subsection{The Proposed Algorithm}\label{subsection:htp}

Comparing the two formulations \eqref{problem:leastsq} and \eqref{eq:conslsint}, experimental results \cite{zhang2017nonconvex,wang2017solving,soltanolkotabi2019structured} suggest that algorithms based on the amplitude equation fitting \eqref{problem:leastsq} are usually more efficient than those on the intensity equation \eqref{eq:conslsint}. Following this, we solve \eqref{problem:leastsq} as well. Our algorithm is motivated by the IHT algorithm \eqref{iter:iht} and the hard thresholding pursuit (HTP) algorithm \cite{foucart2011hard} for compressed sensing.

Let $\bm{x}_k\in\mathbb{R}^n$ be the approximation of $\bm{x}^{\natural}$ at step $k$. We observe that one iteration of \eqref{iter:iht} is just one step of projected gradient descent algorithm (a.k.a. the IHT algorithm \cite{blumensath2008iterative}) applied to the following constrained least squares problem
\begin{equation}\label{min:ls_for_cs}
\mathop{\mathrm{minimize}}\limits_{\|\bm{x}\|_0\leq s}~\frac{1}{2}\|\bm{A}\bm{x}-\bm{y}\odot \mathrm{sgn}\left(\bm{A}\bm{x}_k\right)\lVert_{2}^{2}.
\end{equation}
This formulation is exactly used in compressed sensing to recover an $s$-sparse signal $\bm{x}$ from its linear measurements $\bm{A}\bm{x}=\bm{y}\odot \mathrm{sgn}\left(\bm{A}\bm{x}_k\right)$. Thus, \eqref{iter:iht} can be interpreted as: given $\bm{x}_k$, we first guess the sign of the phaseless measurements $\bm{y}$ by $\mathrm{sgn}\left(\bm{A}\bm{x}_k\right)$, and then we solve the resulting compressed sensing problem \eqref{min:ls_for_cs} by one step of IHT \cite{blumensath2008iterative}. Therefore, to improve the efficiency of \eqref{iter:iht}, we may replace IHT by more efficient algorithms in compressed sensing for solving \eqref{min:ls_for_cs}.

To this end, we use one step of hard thresholding pursuit (HTP) \cite{foucart2011hard} to solve \eqref{min:ls_for_cs}. Given $\bm{x}_k$, there are two sub-steps in HTP. In the first sub-step, HTP estimates the support of the sparse signal by the support of the output of IHT, i.e.,
$$
\sS_{k+1}=\mathrm{supp}\Big(\H_s\left(\bm{x}_{k}+\mu \bm{A}^T\left( \bm{y}\odot \mathrm{sgn}\left(\bm{A}\bm{x}_k\right)-\bm{A}\bm{x}_{k}\right)\right)\Big).
$$
The main computation is a matrix-vector product, and it costs $O(mn)$ operations. In the second sub-step, instead of applying a gradient-type refinement, HTP then solves the least squares in \eqref{min:ls_for_cs} exactly by restricting the support of the unknown on $\sS_{k+1}$, i.e.,
\begin{align}\label{htp:subproblem}
\bm{x}_{k+1}=\mathop{\arg\, \min}\limits_{\mathrm{supp}\left(\bm{x}\right)\subseteq \sS_{k+1}}\frac{1}{2}\lVert \bm{A}\bm{x}-\bm{y}\odot \mathrm{sgn}\left(\bm{A}\bm{x}_k\right)\lVert_{2}^{2}.
\end{align}
This is a standard least squares problem with the coefficient matrix of size $s\times m$, which can be done efficiently in $O(s^2m)$ operations by, e.g., solving the normal equation
\begin{equation}\label{eq:optxk+1}
\bm{A}^{T}_{\sS_{k+1}}\bm{A}_{\sS_{k+1}}[\bm{x}_{k+1}]_{\sS_{k+1}}=\bm{A}^{T}_{\sS_{k+1}}\big(\bm{y}\odot\mathrm{sgn}\left(\bm{A}\bm{x}_k\right)\big),\quad [\bm{x}_{k+1}]_{\sS_{k+1}^c}=\bm{0}.
\end{equation}
Altogether, we obtain the iteration in our proposed algorithm, called HTP for sparse phase retrieval, depicted in Algorithm \ref{alg:htp}. The total computational cost is $O(mn+s^2m)$ per iteration. In case of $s\lesssim \sqrt{n}$, the cost is the same order as $O(mn)$, and thus one iteration of Algorithm \ref{alg:htp} has almost the same computational cost as that of \eqref{iter:iht} and many other popular sparse phase retrieval algorithms.

\begin{algorithm}[hbt!]
   \caption{Hard Thresholding Pursuit (HTP) for Sparse Phase Retrieval}\label{alg:htp}
   \begin{algorithmic}[1]
     \STATE Input: Data $\left\{\bm{a}_i,y_i\right\}_{i=1}^{m}$, step size $\mu>0$ (e.g., $\mu=0.95$).
     \STATE Initialization: Let $\bm{x}_{0}$ be the initial guess produced by, e.g., the spectral initialization given by (Init-1) and (Init-2) in Section~\ref{subsection:global}.
     \STATE $k=0$
     \WHILE{the stopping criteria is not met}
     \STATE $\bm{z}_{k+1}=\bm{A}\bm{x}_k$
     \STATE $\bm{y}_{k+1}=\bm{y} \odot \mathrm{sgn}{\left(\bm{z}_{k+1}\right)}$
     \STATE $\sS_{k+1}=\mathrm{supp}\Big(\H_s\left(\bm{x}_{k}+\mu \bm{A}^T\left( \bm{y}_{k+1}-\bm{z}_{k+1}\right)\right)\Big)$
     \STATE $\bm{x}_{k+1}=\mathop{\arg\, \min}\limits_{\mathrm{supp}\left(\bm{x}\right)\subseteq \sS_{k+1}}\frac{1}{2}\lVert \bm{A}\bm{x}-\bm{y}_{k+1}\lVert_{2}^{2}$
     \STATE $k=k+1$
     \ENDWHILE
     \STATE Output $\bm{x}_k$.
    \end{algorithmic}
\end{algorithm}

HTP has been demonstrated much more efficient than IHT for compressed sensing both theoretically and empirically. Since IHT is a first-order gradient-type algorithm, it converges at most linearly. On the contrary, HTP can break through the barrier of linear convergence, because it is a second-order Newton's method (see e.g. \cite{zhou2019global}). Its acceleration over IHT has been confirmed in many works \cite{blumensath2012accelerated,foucart2011hard}. More interestingly, HTP enjoys a finite-step termination property --- it gives the exact recovery of the underlying sparse signal after at most $O(\log(\|\bm{x}^{\natural}\|_2/|x_{\min}^{\natural}|))$ steps starting from any initial guess provided $\bm{A}$ satisfies the restricted isometry property (RIP), as proved in {\cite[Corollary 3.6]{foucart2011hard}}. Furthermore, the computational cost of HTP per iteration is the same order as that of IHT, if $s$ is sufficiently small compared to $n$. Therefore, HTP outperforms IHT significantly in compressed sensing.

Because the iteration in Algorithm \ref{alg:htp} is HTP for sparse phase retrieval, it is a second-order Newton's algorithm. Other existing nonconvex sparse phase retrieval algorithms are mostly the IHT algorithm \eqref{iter:iht} (for solving \eqref{problem:leastsq}) and \eqref{eq:pwf} (for solving \eqref{eq:conslsint}), and their truncation variants \cite{soltanolkotabi2019structured,cai2016optimal,wang2016sparse}. Those algorithms are first-order gradient-type algorithms. Therefore, according to the results in compressed sensing, our proposed algorithm is expected to require much fewer iterations to achieve an accurate sparse phase retrieval than those existing algorithms. This is indeed true, as shown by one example in Figure~\ref{fig:itererr}. More experimental results are demonstrated in Section \ref{section:experiments}. We see from these experimental results that: as expected, while IHT-type algorithms converge only linearly, the iteration in our proposed Algorithm \ref{alg:htp} converges superlinearly and it gives the exact recovery in just a few of iterations. Moreover, we will prove this theoretically in the next section, revealing that Algorithm \ref{alg:htp} inherits the finite-step convergence property of HTP. Since our proposed algorithm needs the same order of computational cost per iteration as IHT algorithms do when $s\lesssim\sqrt{n}$, Algorithm \ref{alg:htp} is an extremely efficient tool for sparse phase retrieval.

\begin{figure}[!htb]
\centering
{\includegraphics[clip=true,width=0.45\textwidth]{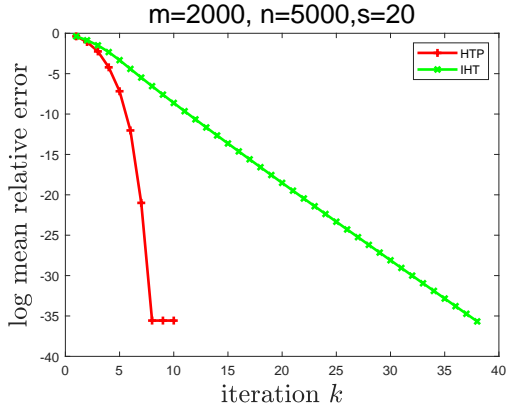}}
\caption{\label{fig:itererr}Sparse phase retrieval: mean relative error ($\log$) vs. iteration number $k$.  Sparsity is fixed to be $s=20$ in this example. The signal dimension $n$ is set to be $5000$, and the sample size $m=2000$. Using the same initialization and step size $\mu=0.75$, the results of IHT (described as~\eqref{iter:iht}) and HTP are shown in the figure. The mean relative error are obtained by averaging $100$ independent trial run. }
\end{figure}

\section{Theoretical Results}\label{section:theory}
When the sparse phase retrieval problem has only one solution up to a global sign, only $\pm\bm{x}^{\natural}$ are global minimizers of the non-convex optimization \eqref{problem:leastsq}. However, due to the non-convexity, no algorithm for solving \eqref{problem:leastsq} is guaranteed automatically to converge to a global minimizer, unless further analysis is provided. This section is devoted to present some theoretical results on the convergence guarantee of Algorithm \ref{alg:htp} to one of the global minimizers $\pm\bm{x}^{\natural}$, and the convergence speed is also investigated. In Section \ref{subsection:local}, we present results on local convergence of Algorithm \ref{alg:htp}. Also, combined with existing results on spectral initialization, we obtain the recovery guarantee of $\pm\bm{x}^{\natural}$ by Algorithm \ref{alg:htp} in Section \ref{subsection:global}.

\subsection{Local Convergence}\label{subsection:local}
We first present our result on the local convergence of Algorithm \ref{alg:htp}. In particular, we show that, when $O(s\log(n/s))$ Gaussian random measurements are used, Algorithm \ref{alg:htp} convergences to the underlying signal $\pm\bm{x}^{\natural}$ (under the metric in \eqref{eq:dist}) if it is initialized in a neighbourhood of $\pm\bm{x}^{\natural}$. More interestingly,Algorithm \ref{alg:htp} is able to return an exact solution after a finite number of steps, while typical local convergence rate of existing provable non-convex sparse phase retrieval algorithms is linear (theoretically). The algorithm finds $\bm{x}^{\natural}$ exactly after at most $O(\log m + \log(\|\bm{x}^{\natural}\|_2/|x_{\min}^{\natural}|))$ steps. The result is summarized in the following Theorem \ref{localconvergence}.

\begin{theorem}[Local convergence]\label{localconvergence}
Let $\{\bm{a}_i\}_{i=1}^m$ be i.i.d. Gaussian random vectors with mean $\bm{0}$ and covariance matrix $\bm{I}$. For any signal $\bm{x}^{\natural}\in\mathbb{R}^{n}$ satisfying $\|\bm{x}^{\natural}\|_0\leq s$, let $\{\bm{x}_k\}_{k\geq 1}$ be the sequence generated by Algorithm \ref{alg:htp} with the input measured data $y_i=|\langle\bm{a}_i,\bm{x}^{\natural}\rangle|$, $i=1,\ldots,m$, the step size $\mu$, and an initial guess $\bm{x}_0$. There exist universal positive constants $\lambda_0, C_0, C_1, C_2, C_3, \mu_1,\mu_2$ and a universal constant $\alpha\in(0,1)$ such that: If
$$
\mu\in[\mu_1,\mu_2],\quad
m\geq C_0s\log(n/s),\quad
\mathrm{dist}\left(\bm{x}_0,\bm{x}^{\natural}\right)\leq\lambda_0\|\bm{x}^{\natural}\|_2,
$$
then
\begin{enumerate}
\item[(a)] With probability at least $1-e^{-C_1 m}$,
$$
\mathrm{dist}\left(\bm{x}_{k+1},\bm{x}^{\natural}\right)\le \alpha \cdot\mathrm{dist}\left(\bm{x}_k,\bm{x}^{\natural}\right),\quad\forall~k\geq 0.
$$
\item[(b)] With probability at least $1-e^{-C_1 m}-m^{1-\beta}$,
$$
\mathrm{dist}\left(\bm{x}_k,\bm{x}^{\natural}\right)=0,
\quad\forall~k> C_2\cdot\max\Big\{\beta\log m,\log\big(\|\bm{x}^{\natural}\|_2/|x_{\min}^{\natural}|\big)\Big\}+C_3,
$$
where $\beta>1$ is arbitrary.
\end{enumerate}
\end{theorem}

The proof of Theorem~\ref{localconvergence} is deferred to Section~\ref{section:proofs}.  We see from Theorem~\ref{localconvergence} that our algorithm not only converges linearly but also enjoys a finite-step termination with exact recovery which depends on the dynamics of the underlying signal. Meanwhile, according to Theorem 5 in \cite{bouchot2016hard}, where a different technique is used, the maximum number of iterations of HTP for compressed sensing would not exceed $2s$ if provided the RIP constant $\delta_{3s}\le \frac{1}{3}$. Thus, the early termination of HTP for compressed sensing is independent of the shape of the underlying signal. On the other hand, the above Theorem \ref{localconvergence} indicates that if the shape of the underlying signal is assumed, one can estimate the upper bound of steps needed for early termination. In fact, many other sparse phase retrieval algorithms \cite{wang2016sparse,netrapalli2013phase} require $x_{\min}^{\natural}$ to be as small as $O(\frac{1}{\sqrt{s}})\|\bm{x}^{\natural}\|_2$. However, this assumption does not hold for many signals, e.g., the $s$-sparse signal decays in form $x^{\natural}_j\sim O(j^{-\gamma})\|\bm{x}^{\natural}\|_2$ for some positive $\gamma$, $j\in \{1,2,\cdots,s\}$ . Thus, if we further assume $ x_{\min}^{\natural}\sim O(s^{-\gamma})\|\bm{x}^{\natural}\|_2$, our algorithm finds an exact global minimizer within $O(\gamma \log m)$ steps for $s\lesssim m$.
Moreover, the computational cost per iteration of our algorithm is in the same order as several matrix-vector products with $\bm{A}$ and $\bm{A}^T$ if $s$ is small compared to $n$. Therefore, the proposed algorithm is efficient.

Now we consider the case that the measurements are not perfect and given by $\bm{y}^{(\varepsilon)}=|\bm{A}\bm{x}^{\natural}|+\bm{\varepsilon}$, where $\bm{\varepsilon}\in \mathbb{R}^m$ is the noise. The following corollary provides a convergence guarantee for HTP in this noisy case, and it can be proved using almost the same argument as Theorem~\ref{localconvergence}. In the corollary, we state that the HTP algorithm is locally stable and robust with respect to additive measurements error.
\begin{corollary}[The noisy case]\label{localconvergence:noisy}
Let $\{\bm{a}_i\}_{i=1}^m$ be i.i.d. Gaussian random vectors with mean $\bm{0}$ and covariance matrix $\bm{I}$. For any signal $\bm{x}^{\natural}\in\mathbb{R}^{n}$ satisfying $\|\bm{x}^{\natural}\|_0\leq s$ and any $\bm{\varepsilon}\in \mathbb{R}^m$, let $\{\bm{x}_k\}_{k\geq 1}$ be the sequence generated by Algorithm \ref{alg:htp} with the input measured data $\bm{y}=\bm{y}^{(\varepsilon)}:=|\bm{A}\bm{x}^{\natural}|+\bm{\varepsilon}$ and an initial guess $\bm{x}_0$. There exist some constant step size $\mu$ and universal positive constants $\lambda_1, C_0, C_1$ and $\alpha_1\in(0,1), d\in (0,3.3)$ such that: If
$$
m\geq C_0s\log(n/s),\quad
\mathrm{dist}\left(\bm{x}_0,\bm{x}^{\natural}\right)\leq\lambda_1\|\bm{x}^{\natural}\|_2,
$$
then with probability at least $1-e^{-C_1 m}$,
$$
\mathrm{dist}\left(\bm{x}_{k+1},\bm{x}^{\natural}\right)\le \alpha_1 \cdot\mathrm{dist}\left(\bm{x}_k,\bm{x}^{\natural}\right)+d\lV\bm{\varepsilon}\rV_2,\quad\forall~k\geq 0.
$$
\end{corollary}
The proof of Corollary~\ref{localconvergence} is deferred to Section~\ref{subsec:proofnoisy}. Therefore, by the local result in Corollary~\ref{localconvergence:noisy}, we know for some small $\lambda_1$, if there is $\mathrm{dist}\left(\bm{x}_0,\bm{x}^{\natural}\right)\leq\lambda_1\|\bm{x}^{\natural}\|_2$ and $m$ at least $O(s\log(n/s))$, then with overwhelming probability we have
\begin{equation*}
\mathrm{dist}\left(\bm{x}_{k},\bm{x}^{\natural}\right)\le \alpha_1^k \lambda_1\|\bm{x}^{\natural}\|_2+d\frac{1-\alpha_1^k}{1-\alpha_1}\lV\bm{\varepsilon}\rV_2,\quad\forall~k\geq 0.
\end{equation*}

\subsection{Initialization and Recovery Guarantees}\label{subsection:global}
To have a recovery guarantee, it remains to design an initial guess $\bm{x}_0$ to satisfy the condition $\mathrm{dist}(\bm{x}_0,\bm{x}^{\natural})\leq \lambda_0\|\bm{x}^{\natural}\|_2$ in Theorem \ref{localconvergence}. The same as many existing phase retrieval algorithms \cite{waldspurger2018phase,chen2019gradient,tan2019online,candes2015phase,chen2015solving,netrapalli2013phase,wang2016sparse,cai2016optimal,jagatap2019sample}, we use a spectral initialization to achieve this goal. The idea of spectral initializations is to construct a matrix whose expectation has $\pm\bm{x}^{\natural}$ as the principal eigenvectors, and thus a principal eigenvector of that matrix is a good approximation to $\pm\bm{x}^{\natural}$.

Consider the case where $\{\bm{a}_i\}_{i=1}^{m}$ are independent Gaussian random vectors. In the standard phase retrieval setting \eqref{problem:PR} without the sparsity constraint, it can be easily shown that the expectation of the matrix $\frac1m\sum_{i=1}^my_i^2\bm{a}_i \bm{a}_i^T$ is $\lV \bm{x}^{\natural}\rV_2^2\cdot\bm{I}+2(\bm{x}^{\natural})(\bm{x}^{\natural})^T$, whose principal eigenvectors are $\pm\bm{x}^{\natural}$. Therefore, we use a principal eigenvector of $\frac1m\sum_{i=1}^my_i^2\bm{a}_i \bm{a}_i^T$ as the initialization. This is the spectral initialization used in, e.g., \cite{candes2015phase}. Other spectral initializations may use principal eigenvectors of variants of $\frac1m\sum_{i=1}^my_i^2\bm{a}_i \bm{a}_i^T$. For example, in the truncated spectral initialization \cite{chen2015solving}, a principal eigenvector of a truncated version of $\frac1m\sum_{i=1}^my_i^2\bm{a}_i \bm{a}_i^T$ is computed to approximate $\bm{x}^{\natural}$ initially. The optimal construction and its asymptotically analysis can be found in \cite{luo2019optimal}.

For sparse phase retrieval, one naive way is to use spectral initializations for standard phase retrieval directly. However, this will need unnecessarily many measurements, and the best sampling complexity expected is $m\sim n$. To overcome this, we have to utilize the sparsity of $\bm{x}^{\natural}$. The idea is to first estimate the support of $\bm{x}^{\natural}$, and then obtain the initial guess by a principal eigenvector of the matrix $\frac1m\sum_{i=1}^my_i^2\bm{a}_i \bm{a}_i^T$ (or its variants) restricted to the estimated support. Though various spectral initialization techniques are valid for our algorithm, we follow a natural strategy introduced in \cite{jagatap2019sample}, which is as in the following.
\begin{enumerate}\label{enum}
\item[(Init-1)] The support of $\bm{x}^{\natural}$ is estimated by the set of indices of top-$s$ values in $\left\{\frac{1}{m}\sum_{i=1}^{m}y_i^2[\bm{a}_{i}]_j^2\right\}_{j=1}^{n}$, denoted by $\tilde{\sS}$. Since the expectation of $\left\{\frac{1}{m}\sum_{i=1}^{m}y_i^2[\bm{a}_{i}]_j^2\right\}_{j=1}^{n}$ is $\big\{\frac1m\big(\|\bm{x}^{\natural}\|_2^2+2(x_j^{\natural})^2\big)\big\}_{j=1}^{n}$, $\tilde{\sS}$ could be a good approximation of the support of $\bm{x}^{\natural}$.
\item[(Init-2)] We let $[\bm{x}_0]_{\tilde{\sS}}$ be a principal eigenvector of $\frac{1}{m}\sum_{i=i}^my_i^2[\bm{a}_{i}]_{\tilde{\sS}}[\bm{a}_{i}]_{\tilde{\sS}}^T$ with length $\|\bm{y}\|_2$, and $[\bm{x}_0]_{\tilde{\sS}^c}=0$. The reason is that $[\pm\bm{x}^{\natural}]_{\tilde{\sS}}$ is the principal eigenvector of the expectation of $\frac{1}{m}\sum_{i=i}^my_i^2[\bm{a}_{i}]_{\tilde{\sS}}[\bm{a}_{i}]_{\tilde{\sS}}^T$, and $\|\bm{x}^{\natural}\|_2$ is the expectation of $\|\bm{y}\|_2$.
\end{enumerate}

This choice of $\bm{x}_0$ indeed satisfies the requirement on $\bm{x}_0$ for any $\lambda_0$ in Theorem \ref{localconvergence}, as stated in \cite[Theorem IV.1]{jagatap2019sample}.

\begin{lemma}[{\cite[Theorem IV.1]{jagatap2019sample}}]\label{lemma:ini}
Let $\{\bm{a}_i\}_{i=1}^m$ be i.i.d. Gaussian random vectors with mean $\bm{0}$ and covariance matrix $\bm{I}$. Let $\bm{x}_0$ be generated by (Init-1) and (Init-2) with input $y_i=|\bm{a}_i^T\bm{x}^{\natural}|$ for $i=1,\ldots,m$, where $\bm{x}^{\natural}\in\mathbb{R}^{n}$ can be any signal satisfying $\|\bm{x}^{\natural}\|_0\leq s$. Then for any $\lambda_0\in(0,1)$, there exist a positive constant $C_4$ depending only on $\lambda_0$ such that if provided $m\ge C_4 s^2\log n$, we have
\begin{align*}
\mathrm{dist}\left(\bm{x}_0,\bm{x}^{\natural}\right)\le \lambda_0\lV\bm{x}^{\natural}\rV_2
\end{align*}
with probability at least $1-8m^{-1}$.
\end{lemma}


Combined with the local convergence theorem, we obtain the recovery guarantee of our proposed Algorithm \ref{alg:htp}.
\begin{theorem}[Recovery Guarantee]\label{globalconvergence}
 Let $\{\bm{a}_i\}_{i=1}^m$ be i.i.d. Gaussian random vectors with mean $\bm{0}$ and covariance matrix $\bm{I}$. For any signal $\bm{x}^{\natural}\in\mathbb{R}^{n}$ satisfying $\|\bm{x}^{\natural}\|_0\leq s$, let $\{\bm{x}_k\}_{k\geq 1}$ be the sequence generated by Algorithm \ref{alg:htp} with the input measured data $y_i=|\langle\bm{a}_i,\bm{x}^{\natural}\rangle|$, $i=1,\ldots,m$, the step size $\mu$, and an initial guess $\bm{x}_0$ generated by (Init-1) and (Init-2). There exist universal positive constants $\mu_1,\mu_2, C_1, C_2, C_3, C_5$ and a universal constant $\alpha\in(0,1)$ such that: If
$$
\mu\in[\mu_1,\mu_2],\quad
m\geq C_5s^2\log n,
$$
then
\begin{enumerate}
\item[(a)] With probability at least $1-e^{-C_1 m}-8m^{-1}$,
$$
\mathrm{dist}\left(\bm{x}_{k+1},\bm{x}^{\natural}\right)\le \alpha\cdot \mathrm{dist}\left(\bm{x}_k,\bm{x}^{\natural}\right),\quad\forall~k\geq 0.
$$
\item[(b)] With probability at least $1-e^{-C_1 m}-9m^{-1}$,
$$
\mathrm{dist}\left(\bm{x}_k,\bm{x}^{\natural}\right)=0,
\quad\forall~k> 2C_2\cdot\max\Big\{\log m,\log\big(\|\bm{x}^{\natural}\|_2/|x_{\min}^{\natural}|\big)\Big\}+C_3.
$$
\end{enumerate}
\end{theorem}
\begin{proof}
It is a direct consequence of Lemma~\ref{lemma:ini} and Theorem~\ref{localconvergence} with $\beta=2$.
\end{proof}

We see that the sampling complexity is $m\sim O(s^2\log n)$ for a recovery guarantee, same as most of existing sparse phase retrieval algorithms and bottlenecked by the initialization. When $s$ is small compared to $n$ (usually, $s\lesssim \sqrt{n}$), this sampling complexity is better than that in general phase retrieval. Therefore, using the sparsity of the underlying signal improves the sampling complexity.\par
In the noisy case $\bm{y}^{(\varepsilon)}=|\bm{A}\bm{x}^{\natural}|+\bm{\varepsilon}$, however, the local results in Corollary~\ref{localconvergence:noisy} can not be extended to global directly. The reason is that the theoretical results in Lemma~\ref{lemma:ini} is only applicable to noise $\bm{\varepsilon}$ that is sufficiently small and distributed according to a scaled sub-exponential random variable (see {\cite[Theorem IV.3]{jagatap2019sample}}). Thus, if combined with the spectral initialization and provided $m\sim O(s^2\log n)$, then the robustness result in Corollary~\ref{localconvergence:noisy} holds in a global sense for small sub-exponential additive noise.

\section{Numerical results and discussions}\label{section:experiments}
In this section, we present some numerical results of Algorithm \ref{alg:htp} and compare it with other existing sparse phase retrieval algorithms.

Throughout the numerical simulation, the target true signal $\bm{x}^{\natural}$ is set to be $s$-sparse whose support are uniformly drawn from all $s$-subsets of $\left\{1,2\cdots,n\right\}$ at random. The nonzero entries of $\bm{x}^{\natural}$ are generated as randn$(s,1)$ in the MATLAB. The sampling vectors $\{\bm{a}_i\}_{i=1}^m$ are i.i.d. random Gaussian vectors with mean $\bm{0}$ and covariance matrix $\bm{I}$. The clean measured data is $\{y_i\}_{i=1}^m$ with $y_i=|\bm{a}_i^T\bm{x}^{\natural}|$ for $i=1,\ldots,m$. The observed data is a noisy version of the clean data defined by
$$
y_i^{(\varepsilon)}=y_i+\sigma\varepsilon_i,\quad i=1,\ldots,m,
$$
where in the noise $\{\varepsilon_i\}_{i=1}^{m}$ are i.i.d. standard Gaussian, and $\sigma>0$ is the standard deviation of the noise. Thus the noise level is determined by $\sigma$.

Our algorithm HTP will be compared with some of the most recent and popular algorithms, including CoPRAM\cite{jagatap2019sample}, Thresholded Wirtinger Flow(ThWF)\cite{cai2016optimal} and SPARse Truncated Amplitude flow (SPARTA)\cite{wang2016sparse}, in terms of efficiency and sampling complexity. In all the experiments, the step size $\mu$ for HTP is fixed to be $0.75$. For SPARTA, the parameters are set to be $\gamma=0.7,\mu=1,\lv\I\rv=\lfloor m/6\rfloor$. The numerical simulation are run on a laptop with $2.6$ GHz quad-core i$7-6700$HQ processor and $8$ GB memory using MATLAB R$2020$a. The relative error between the true signal $\bm{x}^{\natural}$ and its estimation $\hat{\bm{x}}$ is defined by
$$
r\left(\hat{\bm{x}},\bm{x}^{\natural}\right)= \frac{\mathrm{dist}(\hat{\bm{x}},\bm{x}^{\natural})}{\Vert \bm{x}^{\natural}\Vert_2}.
$$
In the experiments, a successful recovery is defined to be in case $r\left(\hat{\bm{x}},\bm{x}^{\natural}\right)\le 10^{-3}$ where $\hat{\bm{x}}$ is the output of an algorithm. Also, recall that in this paper $\log r$ is the logarithm of $r$ to the base $e$ (which is named log relative error in the description of $y$-axis in all the figures). For example, the number $-7$ in the $y$-axis represents $\log r=-7$ or equivalently the relative error $r=e^{-7}$.
\paragraph{Number of iterations and finite-step convergence.}
We test the number of iterations required for the proposed HTP algorithm, by the following three experiments. In the first experiment, we fix the sparsity $s=20$ and the sample size $m=3000$, and let the signal dimension vary as $n=3000,5000, 10000, 20000$. The relative error are obtained by averaging the result of $100$ independent trial runs, and the experimental results are  plotted in Figure~\ref{fig:iterno}. We see that the number of iterations required by our HTP algorithm is very few. More interestingly, we see clearly that the relative error suddenly jumps to almost $0$ after very few iterations, suggesting that our HTP algorithm enjoys an exact recovery in finite steps as predicted by Part (b) of Theorem \ref{localconvergence}. Furthermore, in Figure~\ref{fig:iterno}, the number of iterations for exact recovery does not grow conspicuously while we increase the signal dimension. In the second experiment, we fix the signal dimension and the number of samples $m=n=10000$, and let the sparsity vary $s=10,20,\ldots,100$. Table~\ref{tab:iterno} shows the average maximum number of iteration required for convergence in $100$ independent trial run. We see that the number of iterations needed grows very slowly. Recall Part (b) of Theorem \ref{localconvergence} states that the number of iteration required is $O(\log m+\log(\|\bm{x}^{\natural}\|_2/|x_{\min}^{\natural}|))$. Since the number of $m$ is fixed, this slow growth might due to the growth of $\log(\|\bm{x}^{\natural}\|_2/|x_{\min}^{\natural}|)$ according to our random model generating $\bm{x}^{\natural}$. In the third experiment, we demonstrate the effect of the sample size $m$ on the number of iterations required. We fix the underlying signal dimension $n=10000$ and sparsity $s=20$, and let $m$ vary from $2000$ to $10000$ at every $1000$. The results are given in Table~\ref{tab:iterno2}. We see that the number of iterations even decays very slowly with respect to $m$, suggesting the dependency on $m$ in Part (b) in Theorem \ref{localconvergence} is somewhat conservative.

\begin{figure}[!htb]
\centering
\subfigure[Dimension $n=3000$]{\includegraphics[trim =0.5cm 0cm 0.5cm 0cm,clip=true,width=0.33\textwidth]{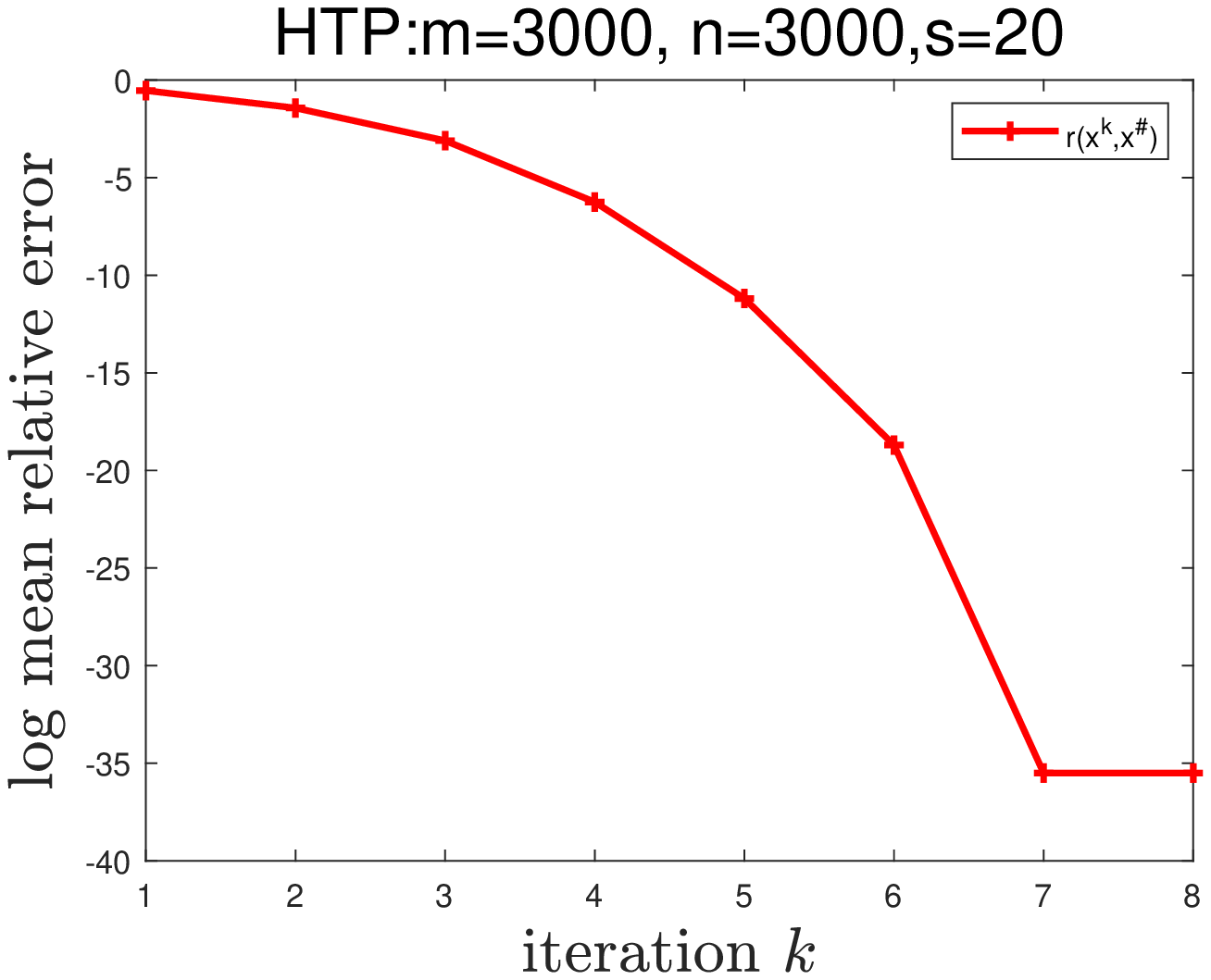}}
\subfigure[Dimension $n=5000$]{\includegraphics[trim =0.5cm 0cm 0.5cm 0cm,clip=true,width=0.33\textwidth]{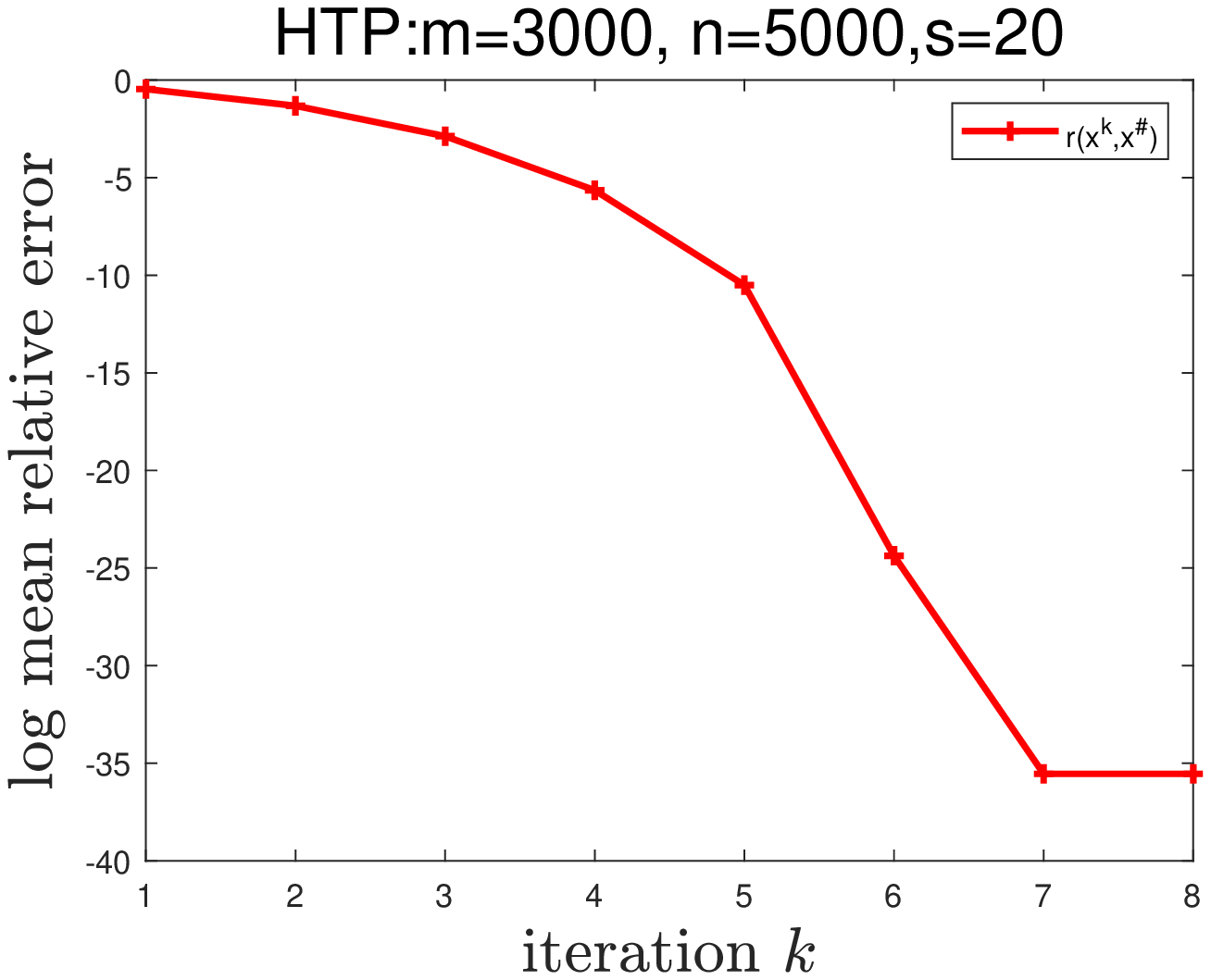}}
\subfigure[Dimension $n=10000$]{\includegraphics[trim =0.5cm 0cm 0.5cm 0cm,clip=true,width=0.33\textwidth]{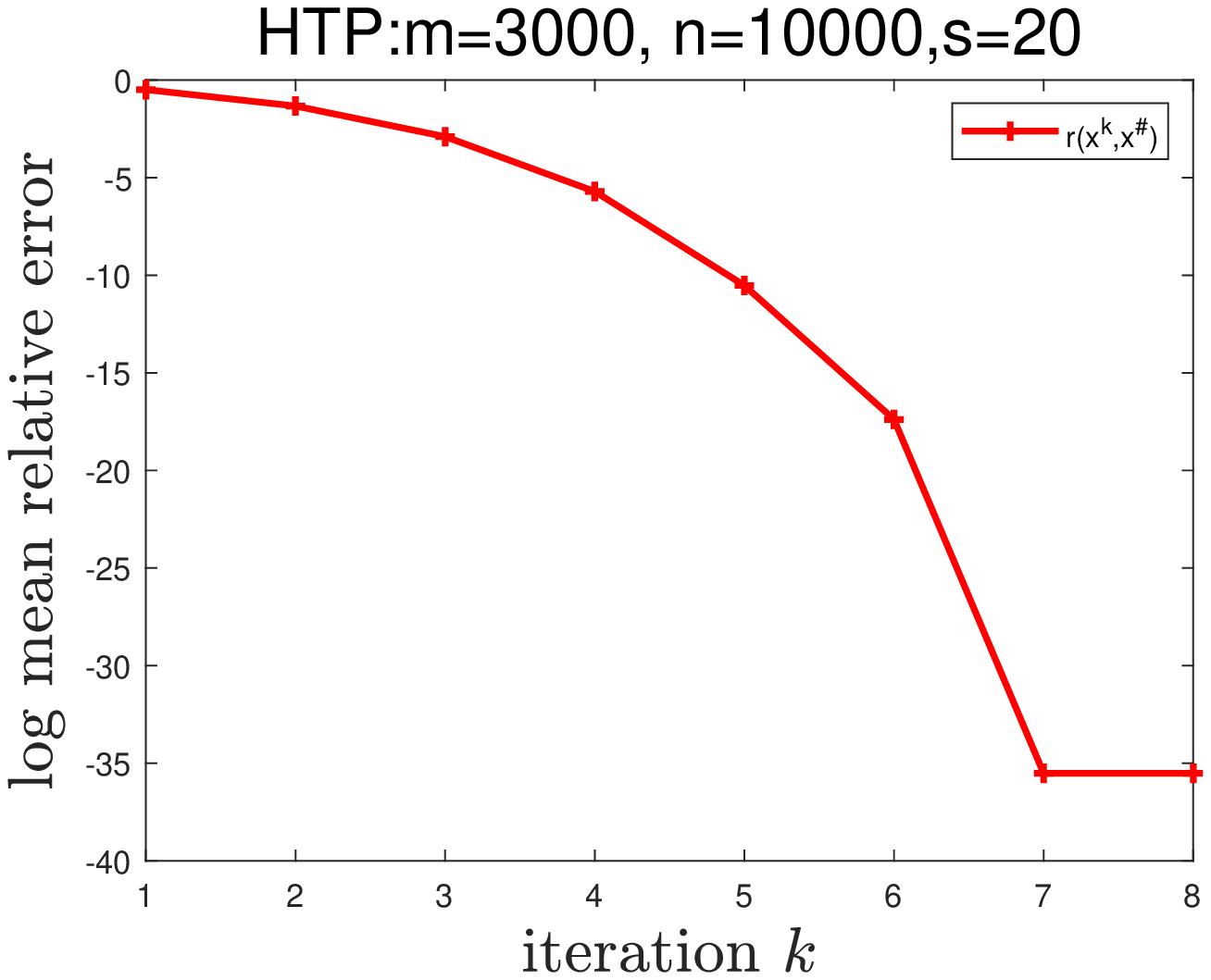}}
\subfigure[Dimension $n=20000$]{\includegraphics[trim =0.5cm 0cm 0.5cm 0cm,clip=true,width=0.33\textwidth]{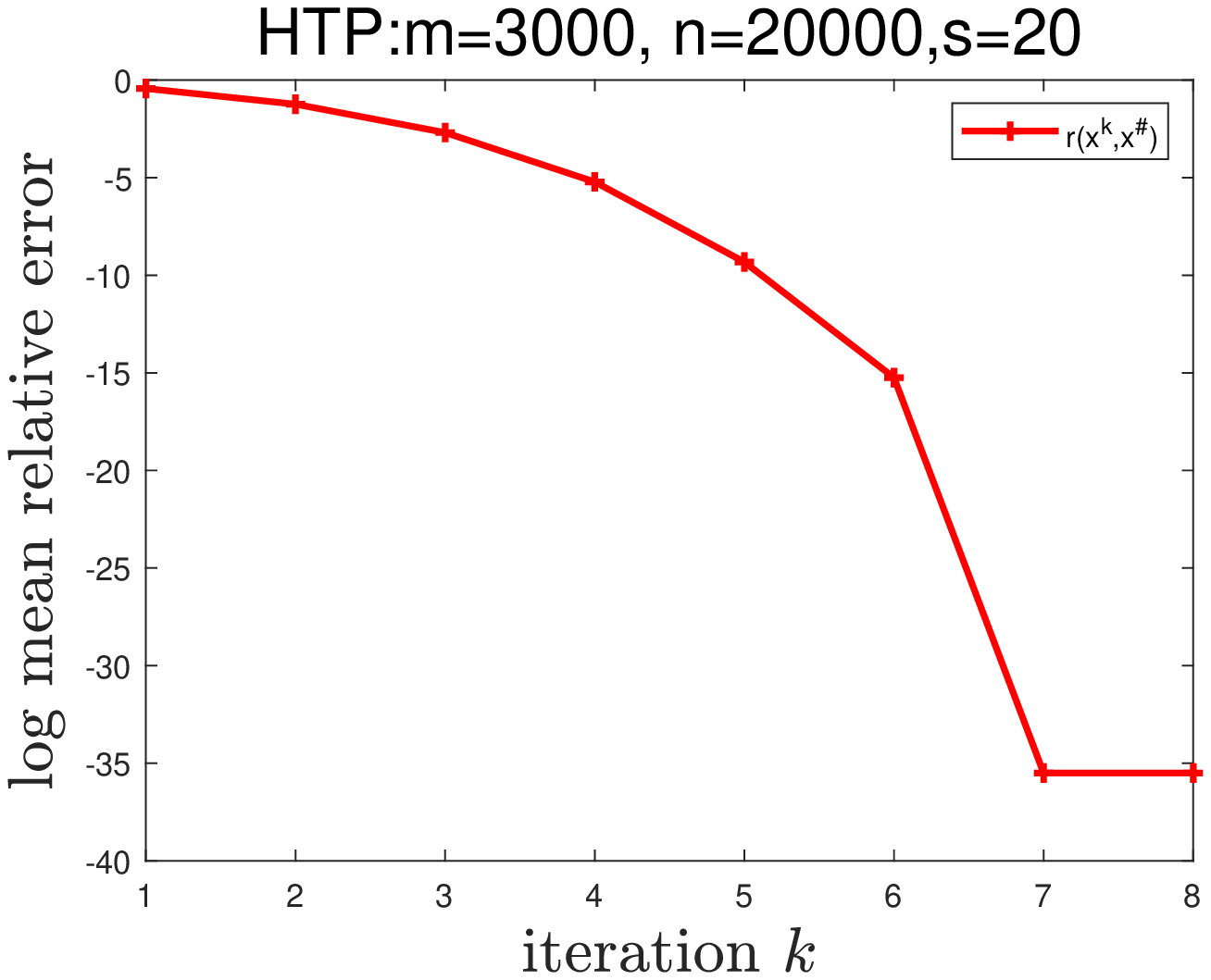}}
\caption{\label{fig:iterno}Log mean relative error vs. the iteration number $k$ for various signal dimensions. The sparsity $s$ and the sample size $m$ are fixed $s=20$ and $m=3000$ respectively. From (a) to (d): the signal dimension $n$ is set to be $3000, 5000, 10000, 20000$, respectively. The mean relative error are obtained by averaging $100$ independent trial runs.}
\end{figure}
\begin{table}[htb!]
\centering
\caption{Effect of the sparsity $s$ on the number of iterations required. The signal dimension $n$ and the sampling size $m$ are fixed $n=m=10000$. The Max iteration no. are obtained by the maximum number of iterations required for convergence ($r\left(\hat{\bm{x}},\bm{x}^{\natural}\right)\le 10^{-10}$) in $100$ independent trial runs.}\label{tab:iterno}

\begin{tabular}{cccccccccccc}
\hline\hline
   &Sparsity $s$   &10   &20   &30   &40  &50  &60  &70   &80 &90  &100  \\
 \hline
&Max iteration no.  &5    &6   &6    &6   &7    &7   &7   &7    &8  &8   \\
\hline\hline
\end{tabular}
\end{table}
\begin{table}[htb!]
\centering
\caption{Effect of the sample size $m$ on the number of iterations required. The signal dimension $n$ and the sparsity $s$ are fixed to be $n=10000$ and $s=20$. The Max iteration no. are obtained by the maximum number of iterations required for convergence ($r\left(\hat{\bm{x}},\bm{x}^{\natural}\right)\le 10^{-10}$) in $100$ independent trial runs.}\label{tab:iterno2}
\begin{tabular}{cccccccccccc}
\hline\hline
   &Sample size $m$  &2000   &3000   &4000   &5000  &6000  &7000  &8000   &9000 &10000   \\
 \hline
&Max iteration no.   &8      &7        &6     &6     &6     &6     &6      &6     &6   \\
\hline\hline
\end{tabular}
\end{table}

\paragraph{Running time comparison.}
We compare our HTP algorithm with existing sparse phase retrieval algorithms, including CoPRAM, HTP, ThWF and SPARTA, in terms of running time. The signal dimension $n$ is fixed to be $3000$. The comparison are demonstrated by two experiments. In the first experiment, the sparsity $s$ is set to be $20,30$, and the sample size $m$ is fixed to be $2000$ to ensure a high successful recovery rate (see Figure~\ref{PhaseTranss30}). The noise level $\sigma$ is set to be $0,0.01,0.05$ respectively. We plot in Figure~\ref{fig:itertime} the results of averages $100$ trial runs with those fail trials ignored. In the second experiment, the sparsity $s$ is set to be $10,20,30,40$ respectively. In Figure~\ref{cputime}, we plot the running time required for a successful signal recovery (in the sense of $r\left(\hat{\bm{x}},\bm{x}^{\natural}\right)\le 10^{-3}$). All the mean value are obtained by averaging $100$ independent trial runs with those failing trials filtered out. From both figures, we see that our proposed HTP algorithm is the fastest among all tested algorithms, and it is at least 2 times faster than other algorithms to achieve a recovery of relative error $10^{-3}$. Furthermore, the running time of our HTP algorithm grows the least with respect to $m$ among all tested algorithms.

\begin{figure}[!htb]
\centering
\subfigure[Noise free, sparsity $s=20$]{\includegraphics[clip=true,width=0.35\textwidth]{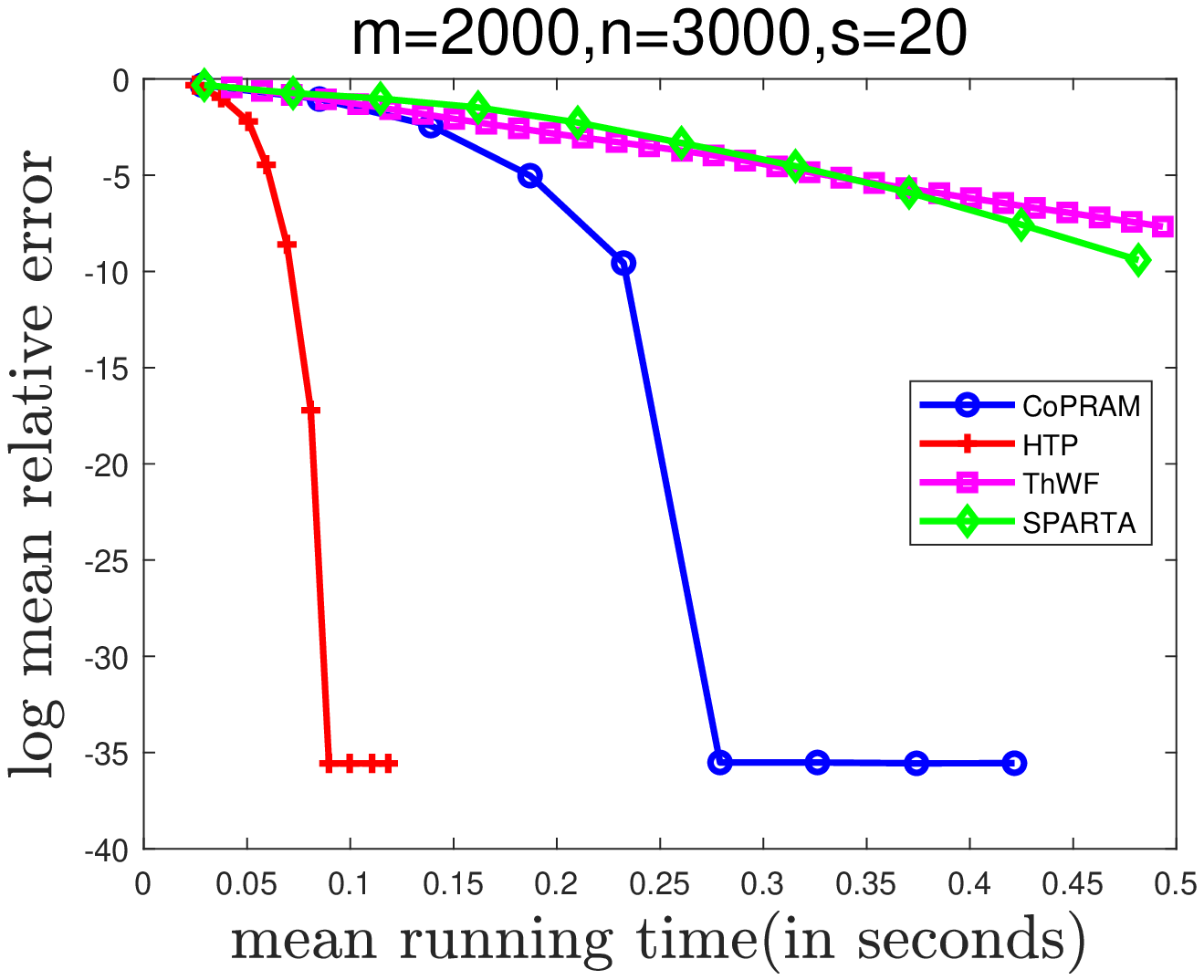}}
\subfigure[Noise free, sparsity $s=30$]{\includegraphics[clip=true,width=0.35\textwidth]{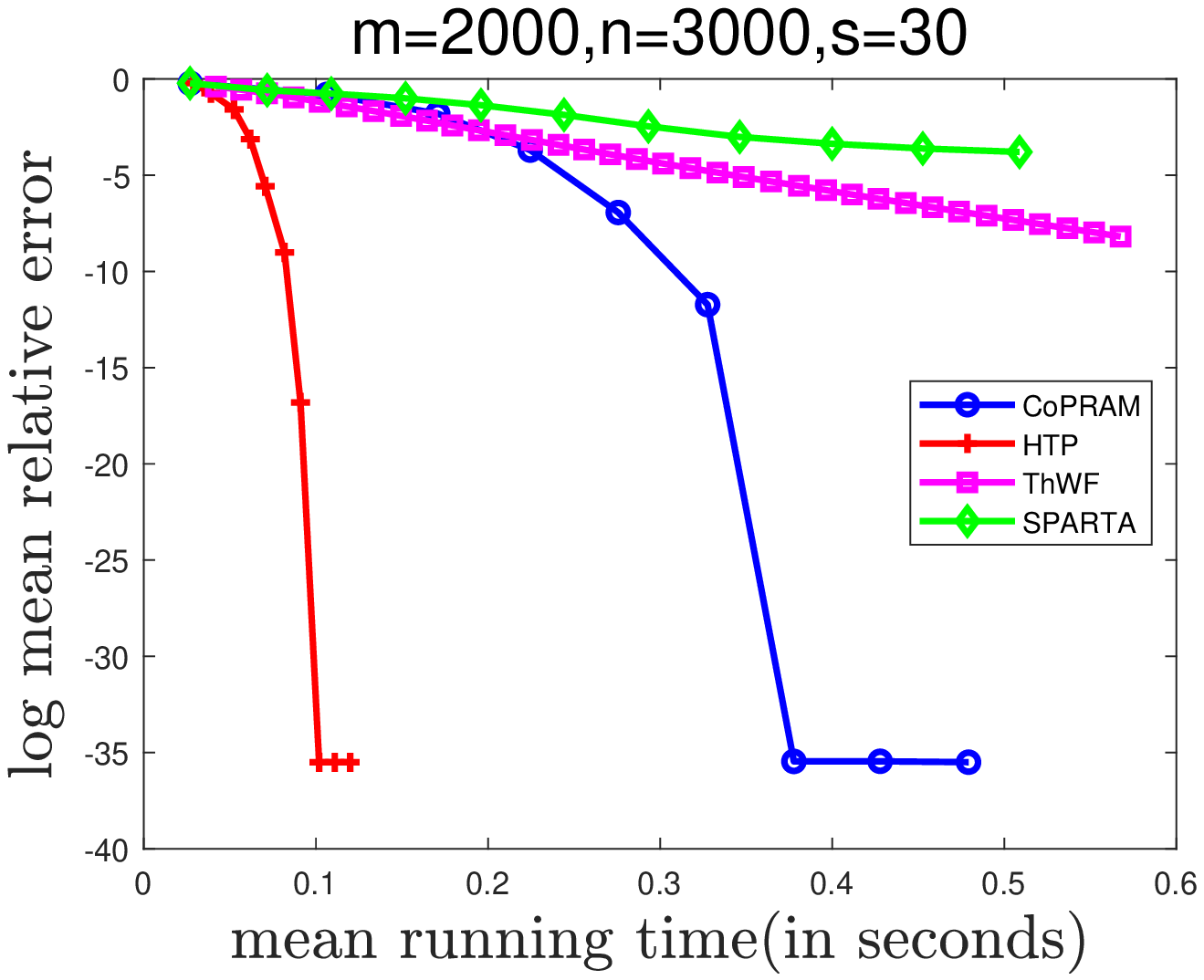}}
\subfigure[Noise level $\sigma=0.01$, sparsity $s=20$]{\includegraphics[clip=true,width=0.35\textwidth]{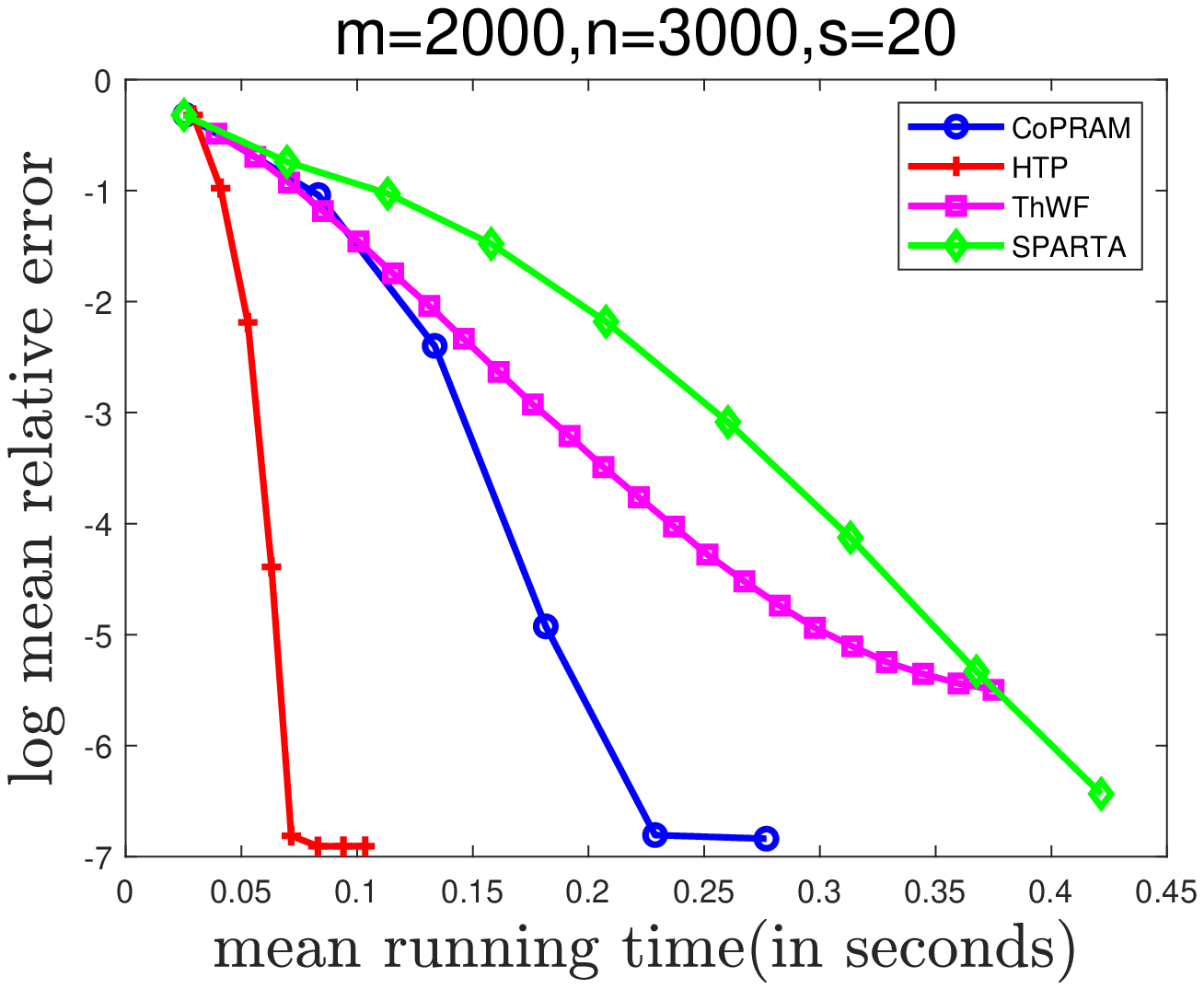}}
\subfigure[Noise level $\sigma=0.01$, sparsity $s=30$]{\includegraphics[clip=true,width=0.35\textwidth]{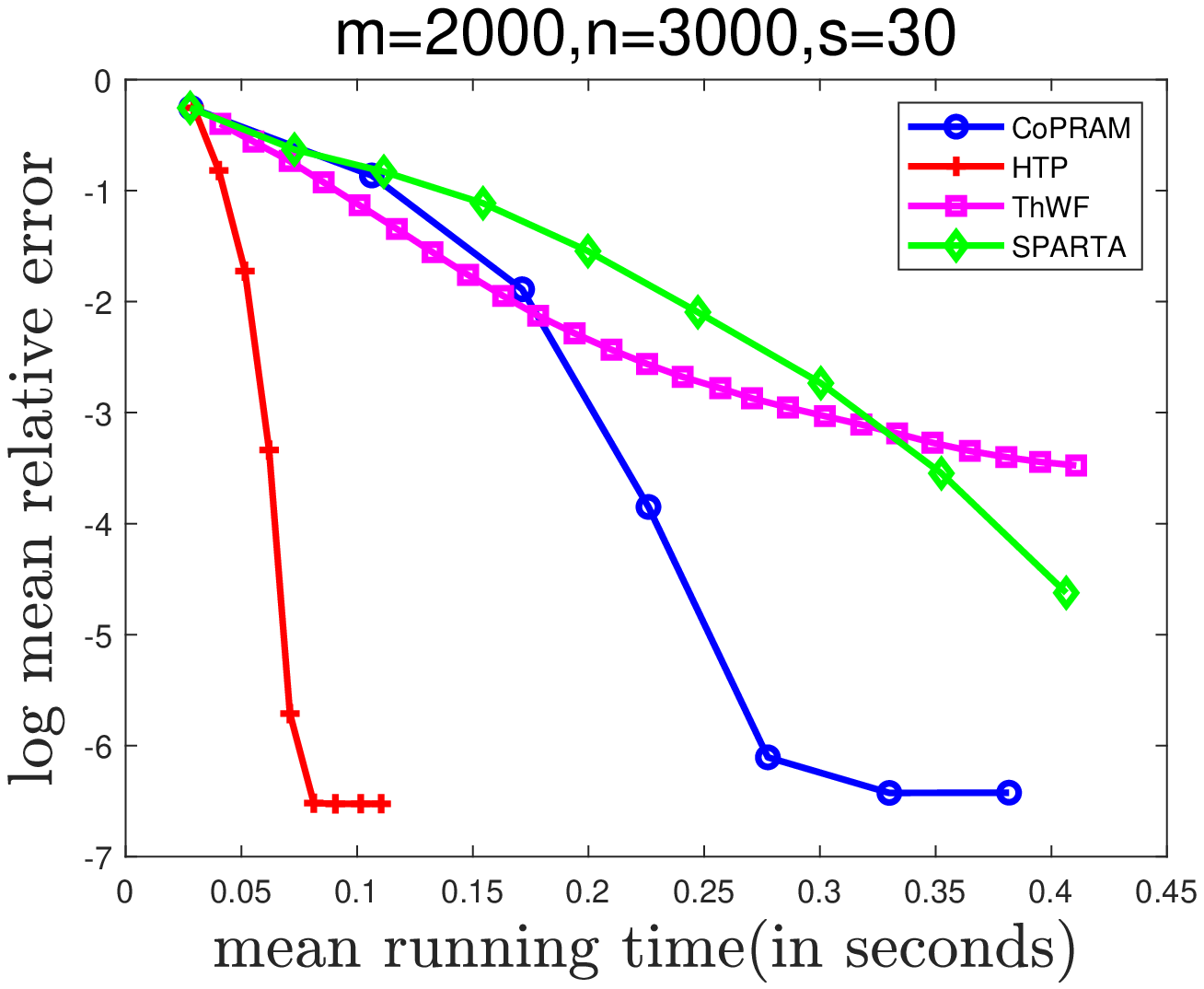}}
\subfigure[Noise level $\sigma=0.05$, sparsity $s=20$]{\includegraphics[clip=true,width=0.35\textwidth]{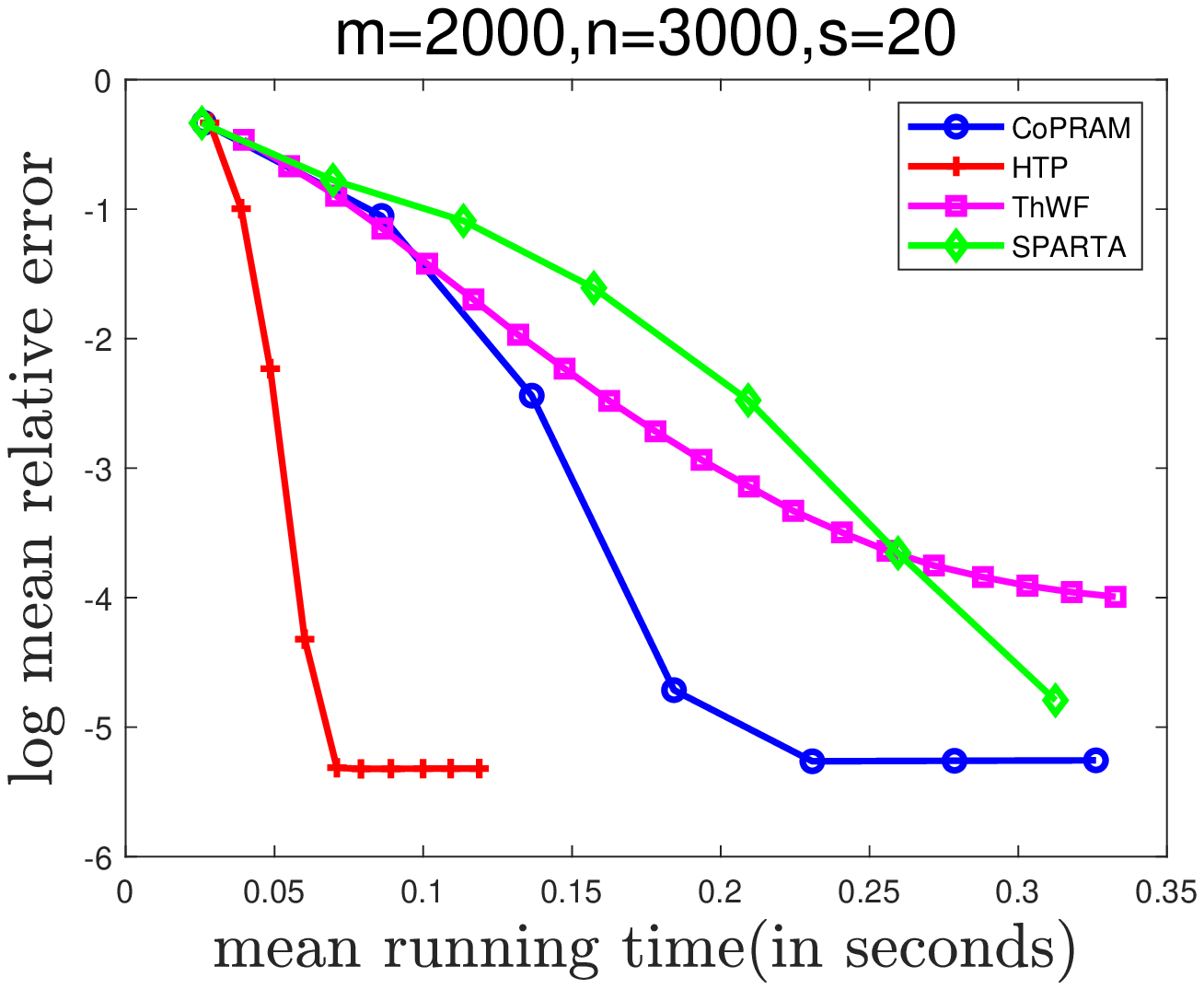}}
\subfigure[Noise level $\sigma=0.05$, sparsity $s=30$]{\includegraphics[clip=true,width=0.35\textwidth]{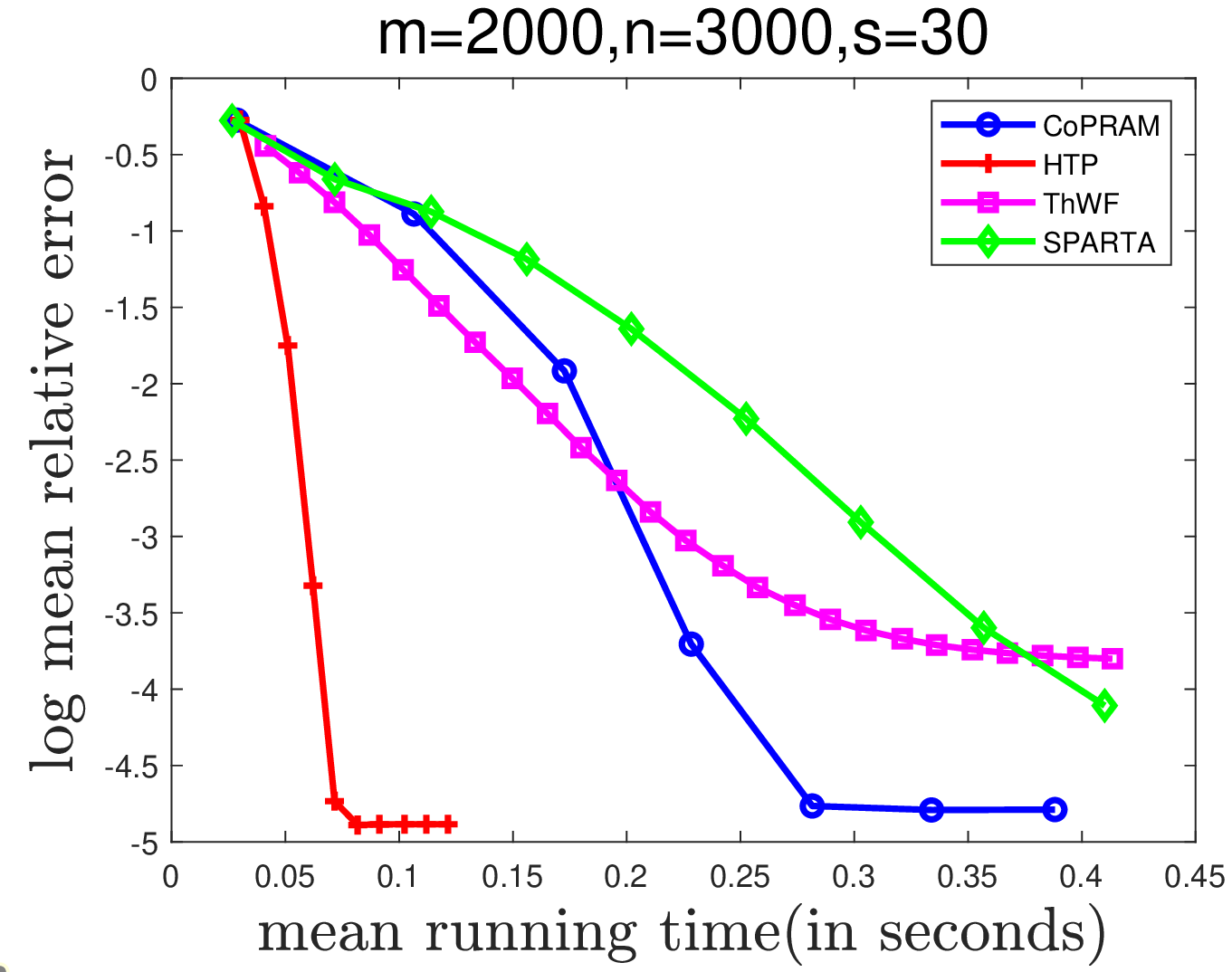}}
\caption{\label{fig:itertime}Log mean relative error vs. mean running time for different algorithms. The signal dimension $n$ and the sample size $m$ are fixed to be $n=3000, m=2000$. From left to right: the sparsity is set to be $s=20$ and $s=30$ respectively. From top to bottom: the noise level $\sigma$ is set to be $\sigma=0$, $\sigma=0.01$, and $\sigma=0.05$ respectively. The results are averages of the corresponding results of $100$ independent trial runs, ignoring the fail trials.}
\end{figure}

\begin{figure}[!htb]
\centering
\subfigure[Sparsity s=10]{\includegraphics[clip=true,width=0.39\textwidth]{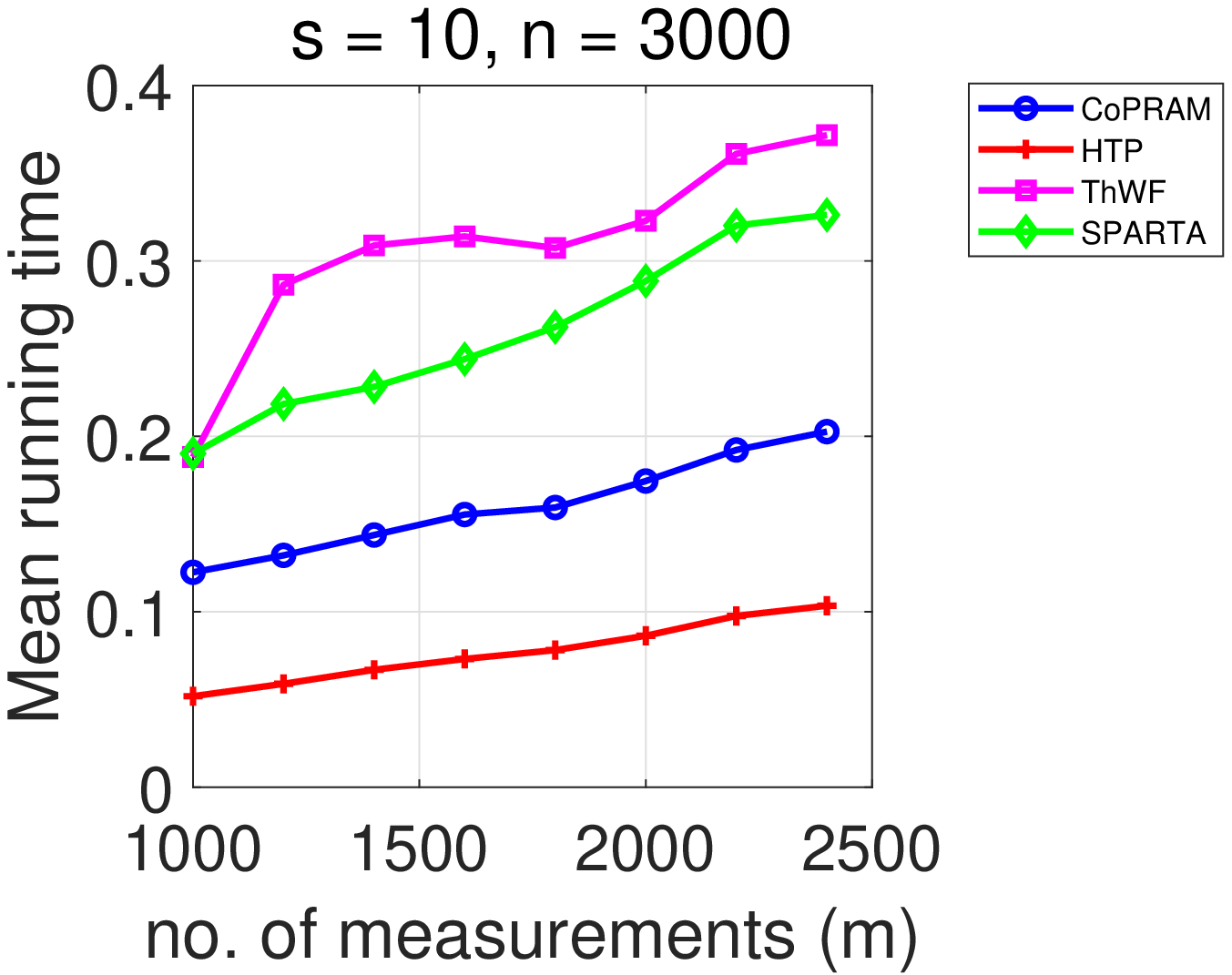}}
\subfigure[Sparsity s=20]{\includegraphics[clip=true,width=0.39\textwidth]{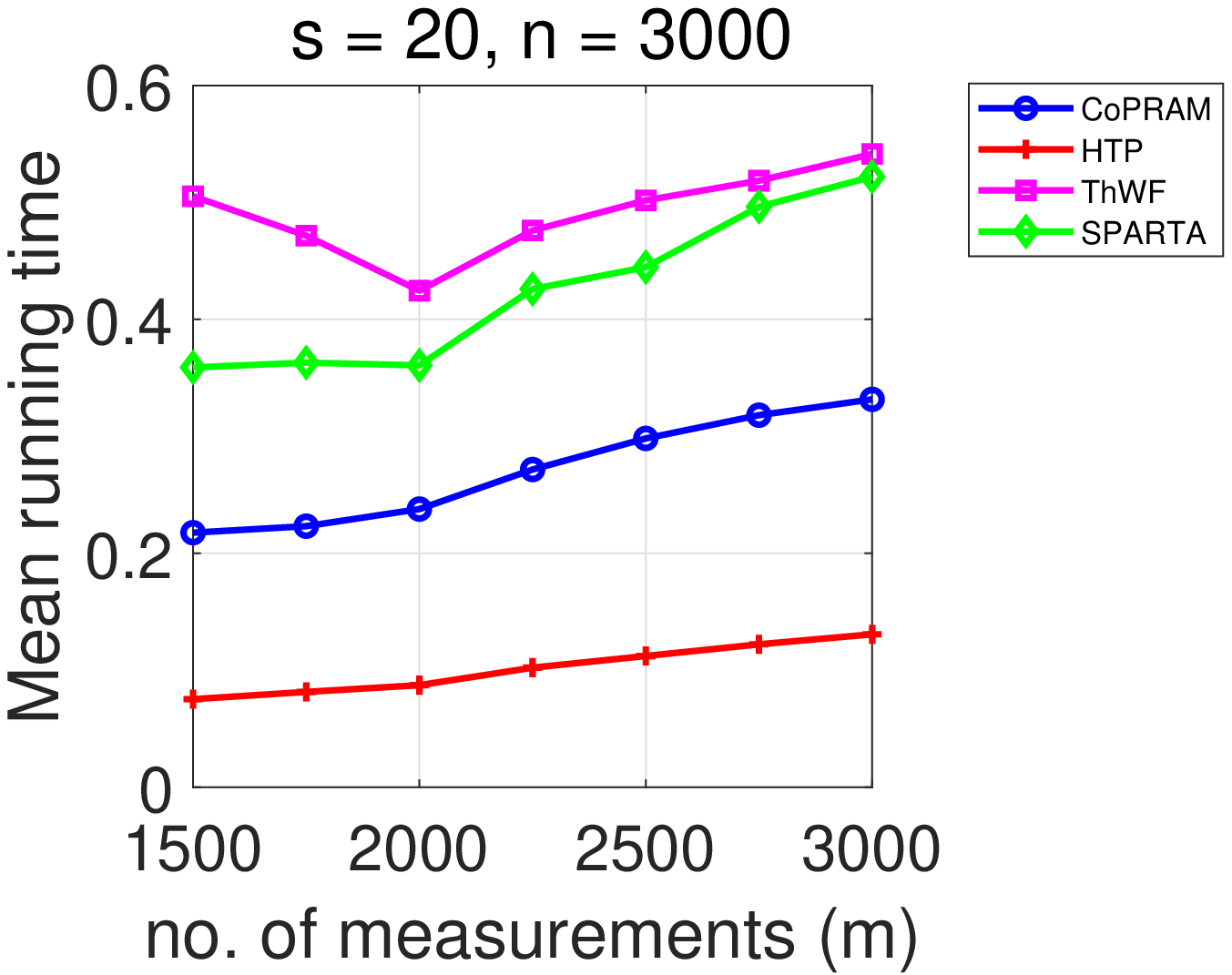}}
\subfigure[Sparsity s=30]{\includegraphics[clip=true,width=0.39\textwidth]{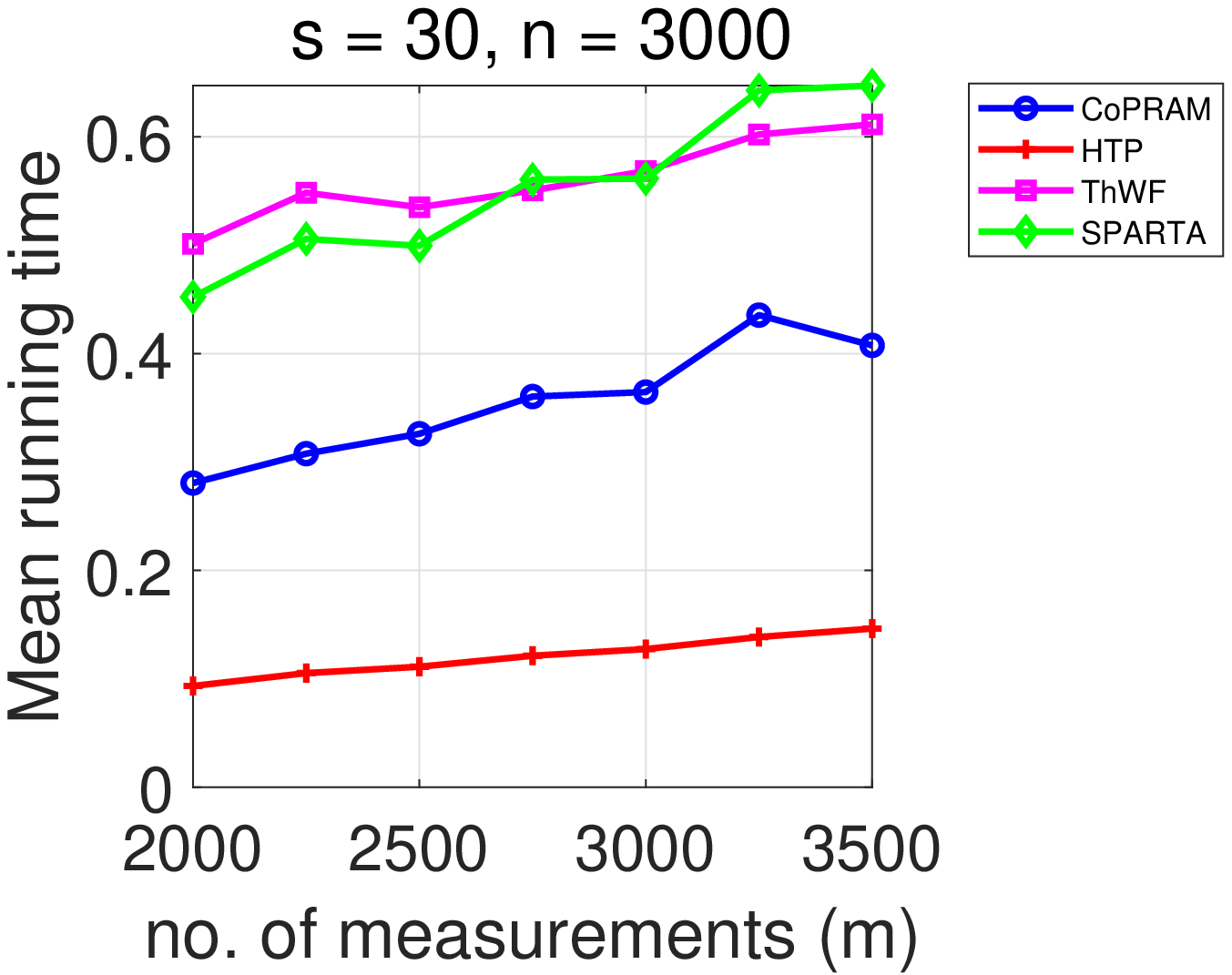}}
\subfigure[Sparsity s=40]{\includegraphics[clip=true,width=0.39\textwidth]{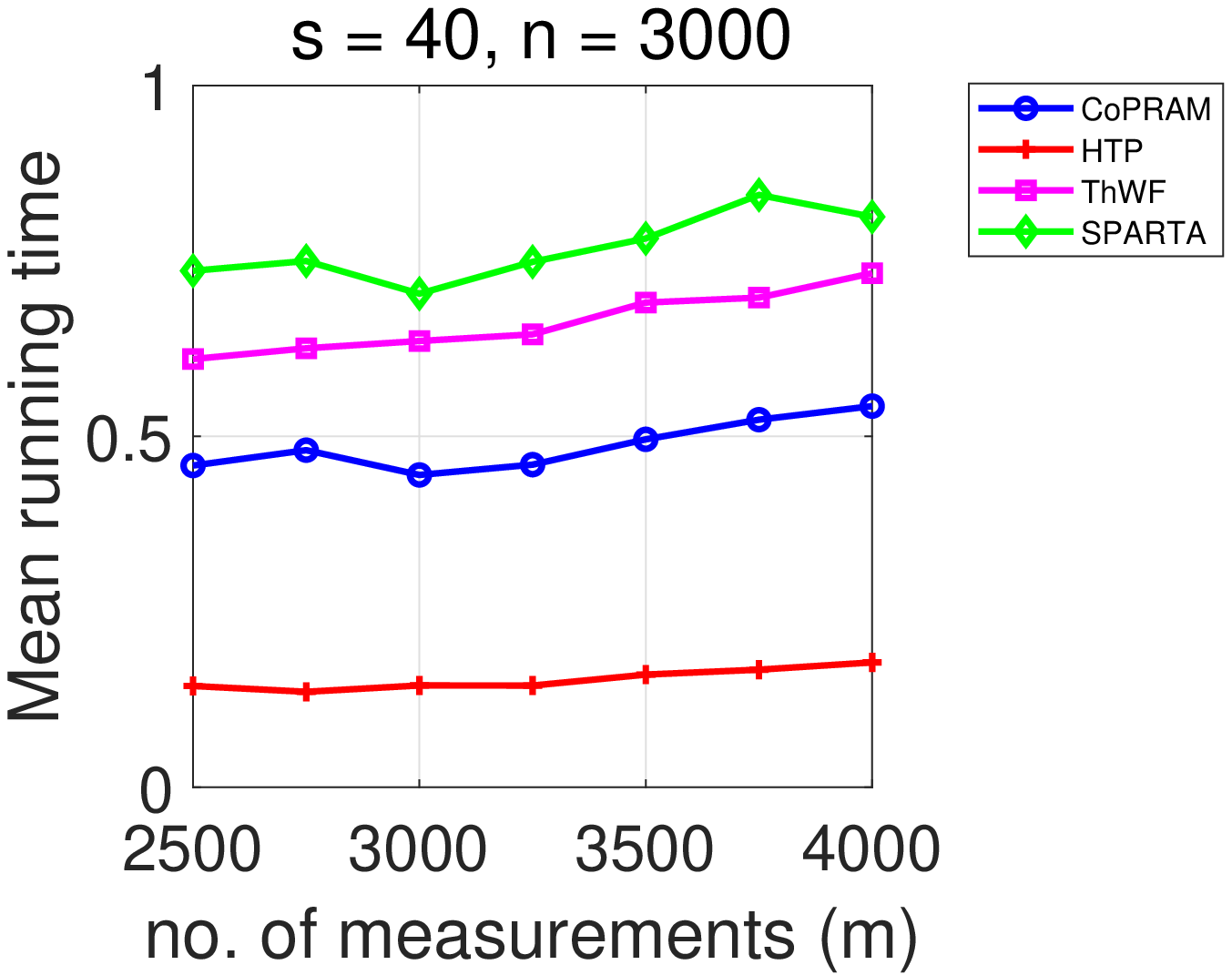}}
\caption{\label{cputime}Mean running time (in seconds) required for a successful recovery ($r\left(\hat{\bm{x}},\bm{x}^{\natural}\right)\le 10^{-3}$) with various sample size for different algorithms. The signal dimension is fixed to be $n=3000$. From (a) to (d): the sparsity $s$ is set to be $10, 20, 30, 40$, respectively. The results are obtained by averaging $100$ independent trial runs with those fail trials filtered out.}
\end{figure}

\paragraph{Robustness to noise.}
We show the effect of noise on the recovery error. In this experiment, we choose $n=3000$, $m=2000$, $s=30$, and we test our HTP algorithm under different noise level of the observed measured data. We plot the mean relative error by our algorithm against the signal-to-noise ratios of the measured data in Figure~\ref{noiseRobust}. The mean relative error are obtained by averaging $100$ independent trial runs. We see that the log relative error decays almost linearly with respect to the SNR in dB, suggesting that the recovery error of the HTP algorithm is controlled by a constant times the noise level in $\bm{y}$. Therefore, our proposed algorithm is robust to noise contained in the observed phaseless measured data.

\begin{figure}[!htb]
\centering
{\includegraphics[trim =0cm 0cm 1.0cm 0cm,clip=true,width=0.36\textwidth]{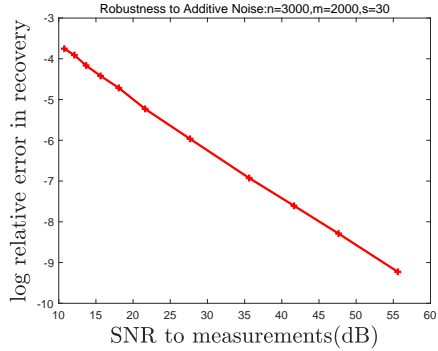}}
\caption{\label{noiseRobust} Robustness to additive Gaussian noise. We set $n=3000, m=2000,s=30$. The $y$-axis is the log mean relative error in the recovery by HTP, and the $x$-axis is the signal-to-noise ratios (SNR) of the measurements data. The results are obtained by averaging $100$ independent trial runs.}
\end{figure}

\paragraph{Phase transition.}
Now we show the phase transition of our HTP algorithm and compare it with other algorithms. In this experiment, we fix the signal dimension $n=3000$. First, for the sparsity $s=20$ and $s=30$, the successful recovery rate are shown in Figure~\ref{PhaseTranss30} when the sample size $m$ vary from $250$ to $3000$. Moreover, Figure \ref{PhaseTrans} depicts the success rate of different algorithms with different sparsities $s$ and different sample sizes $m$: the sparsity $s$ shown in the $y$-axis vary from $10$ to $80$ with grid size $5$, and sample size $m$ shown in the $x$-axis vary from $250$ to $3000$ with grid size $250$. In the figure, the grey level of a block means the successful recovery rate: black means $0\%$ successful reconstruction, white means $100\%$ successful reconstruction, and grey means a successful reconstruction rate between $0\%$ and $100\%$. The successful recovery rates are obtained by $100$ independent trial runs. From the figure, we see that our HTP algorithm, CoPRAM, and SPARTA have similar phase transitions while all algorithms are comparable. For smaller $s$, HTP, CoPRAM, and SPARTA are slightly better than ThWF. For larger $s$, ThWF is slightly better than the others.

\begin{figure}[!htb]
\centering
\subfigure[$s=20$]{\includegraphics[clip=true,width=0.38\textwidth]{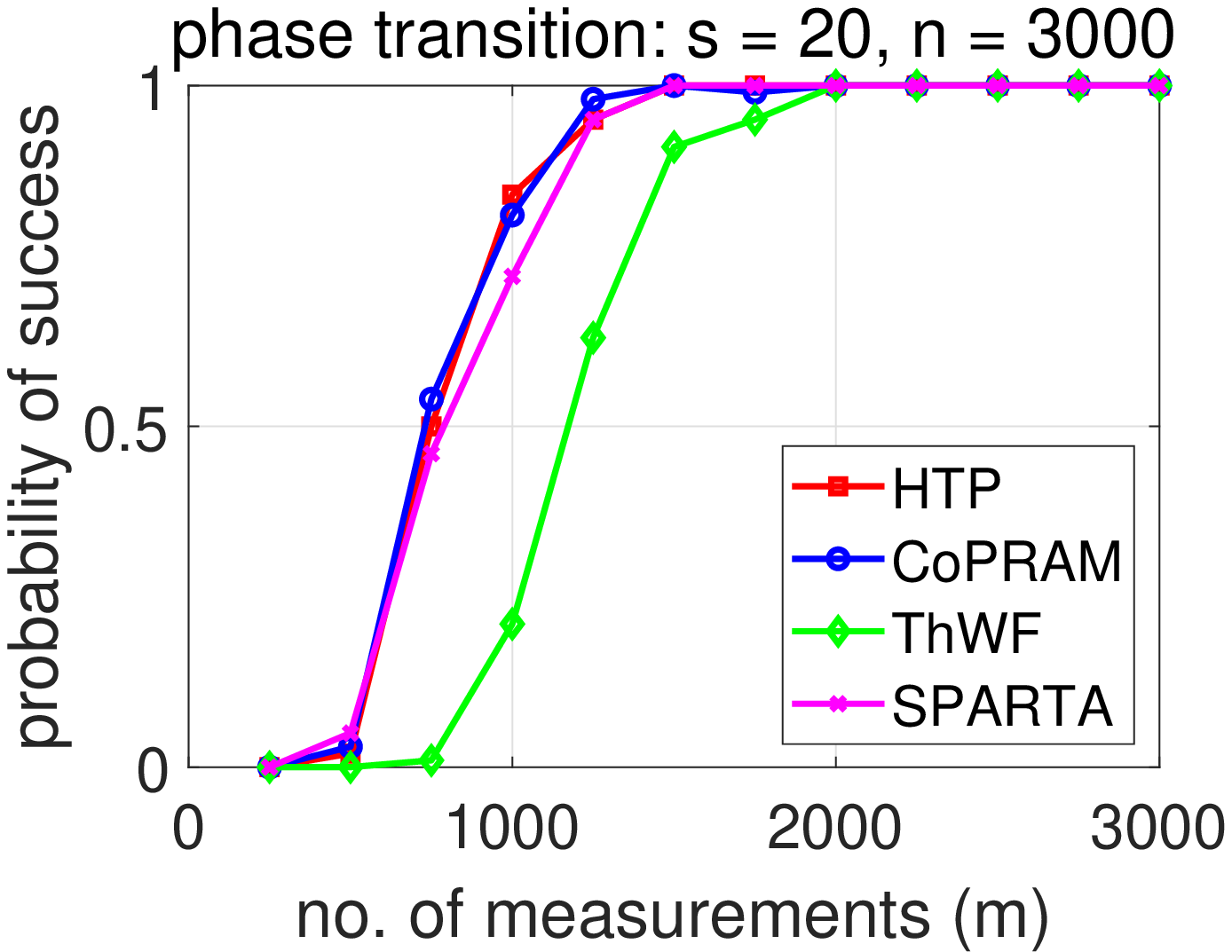}}
\subfigure[$s=30$]{\includegraphics[clip=true,width=0.38\textwidth]{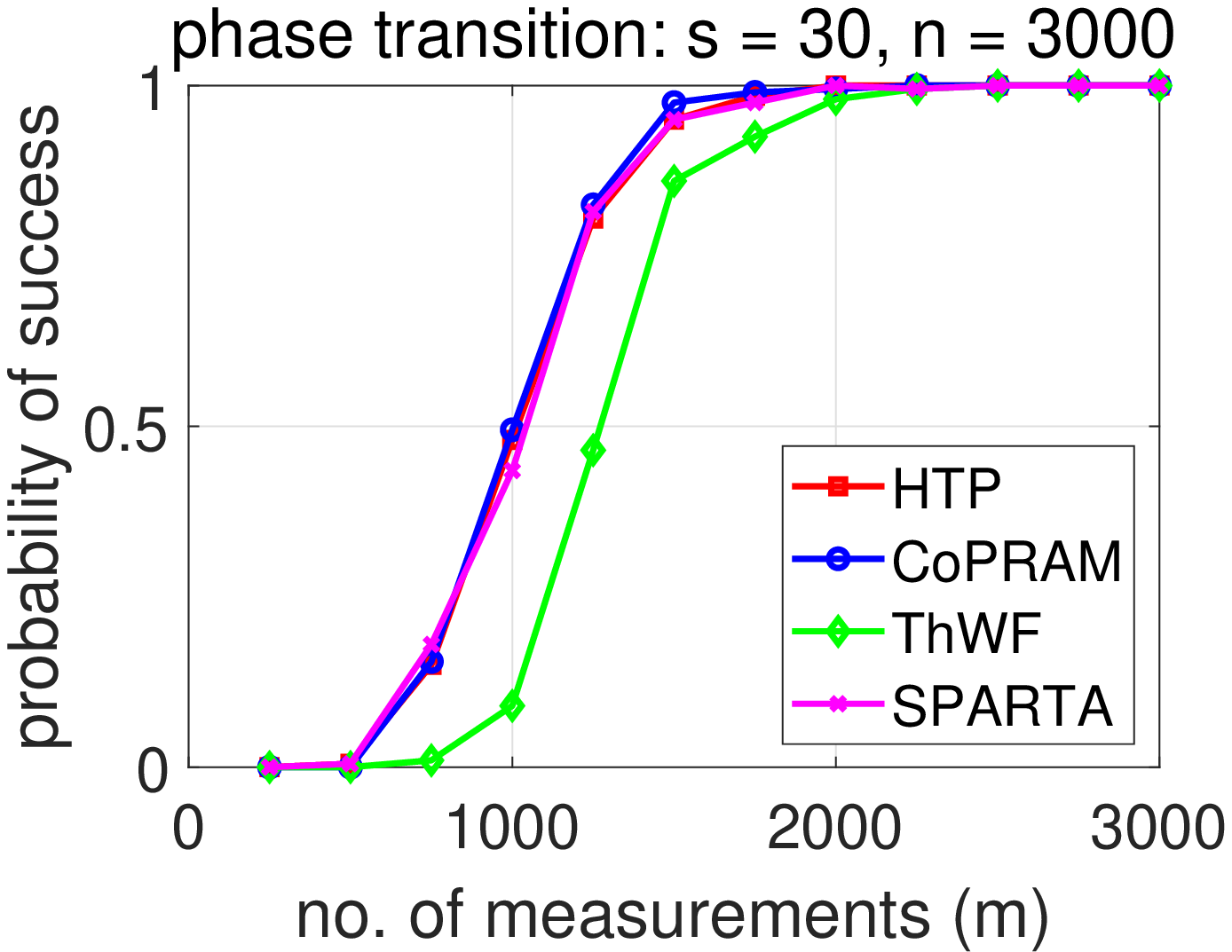}}
\caption{\label{PhaseTranss30}Phase transition for different algorithms with signal dimension $n=3000$ and $s=20,30$.}
\end{figure}
\begin{figure}[!htb]
\centering
\subfigure[ThWF]{\includegraphics[trim =1.5cm 0cm 1.5cm 0cm,clip=true,width=0.24\textwidth]{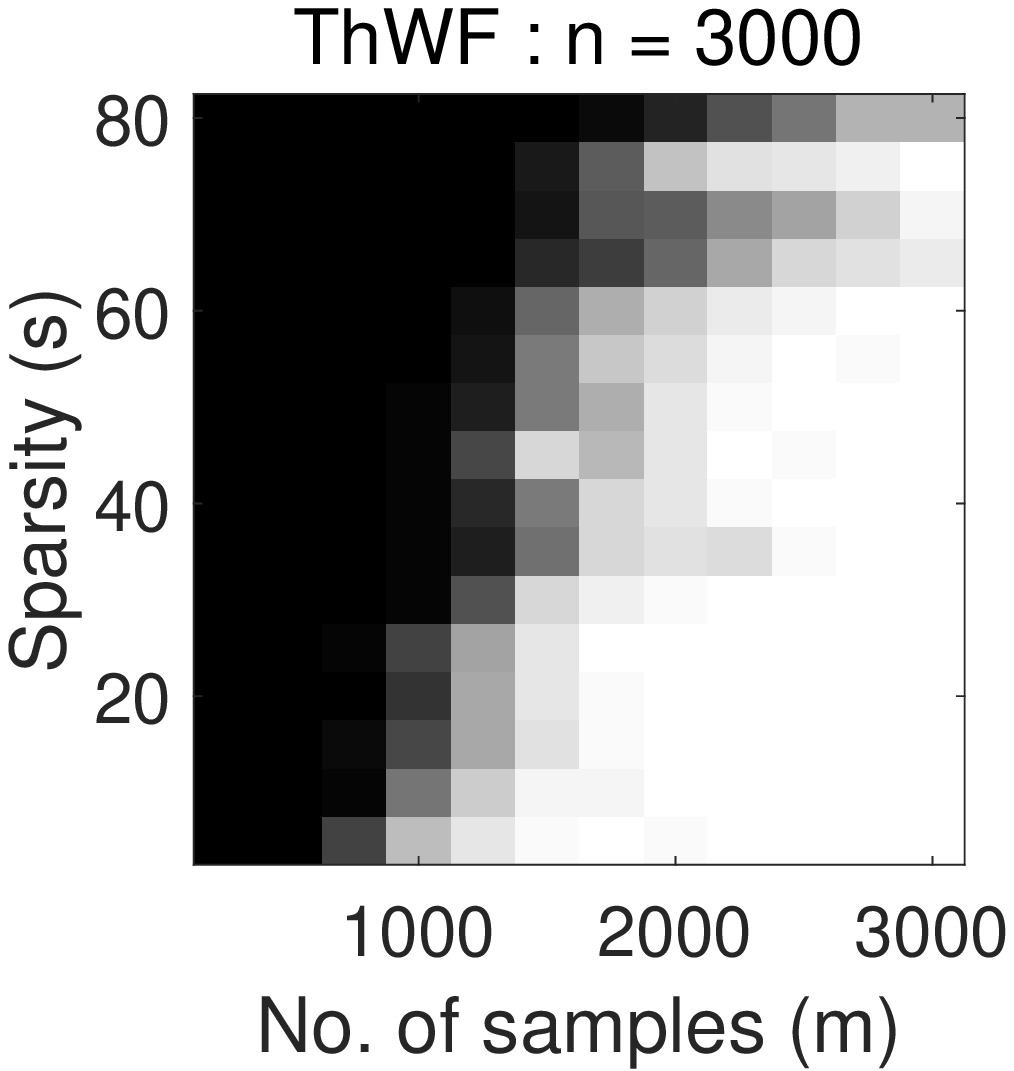}}
\subfigure[CoPRAM]{\includegraphics[trim =1.5cm 0cm 1.5cm 0cm,clip=true,width=0.24\textwidth]{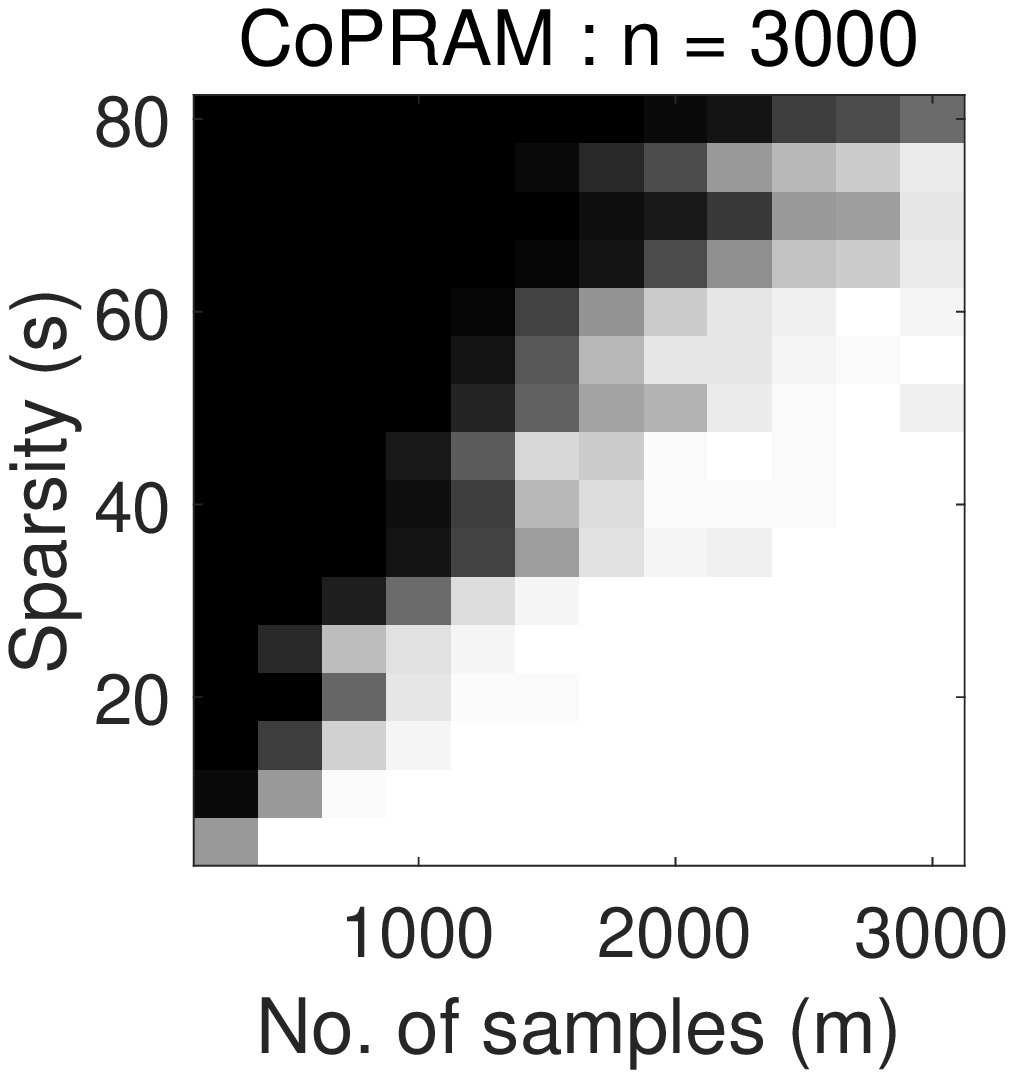}}
\subfigure[SPARTA]{\includegraphics[trim =1.5cm 0cm 1.5cm 0cm,clip=true,width=0.24\textwidth]{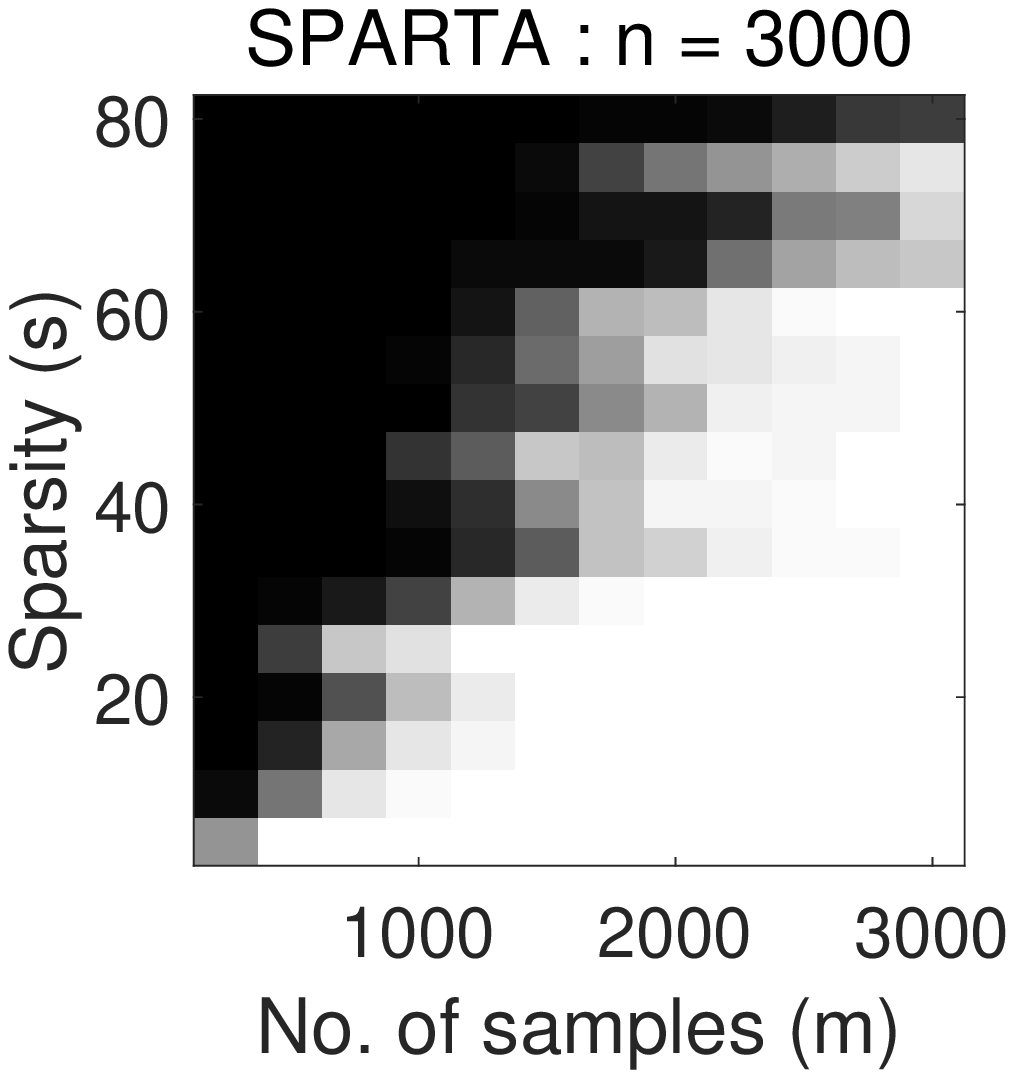}}
\subfigure[HTP]{\includegraphics[trim =1.5cm 0cm 1.5cm 0cm,clip=true,width=0.24\textwidth]{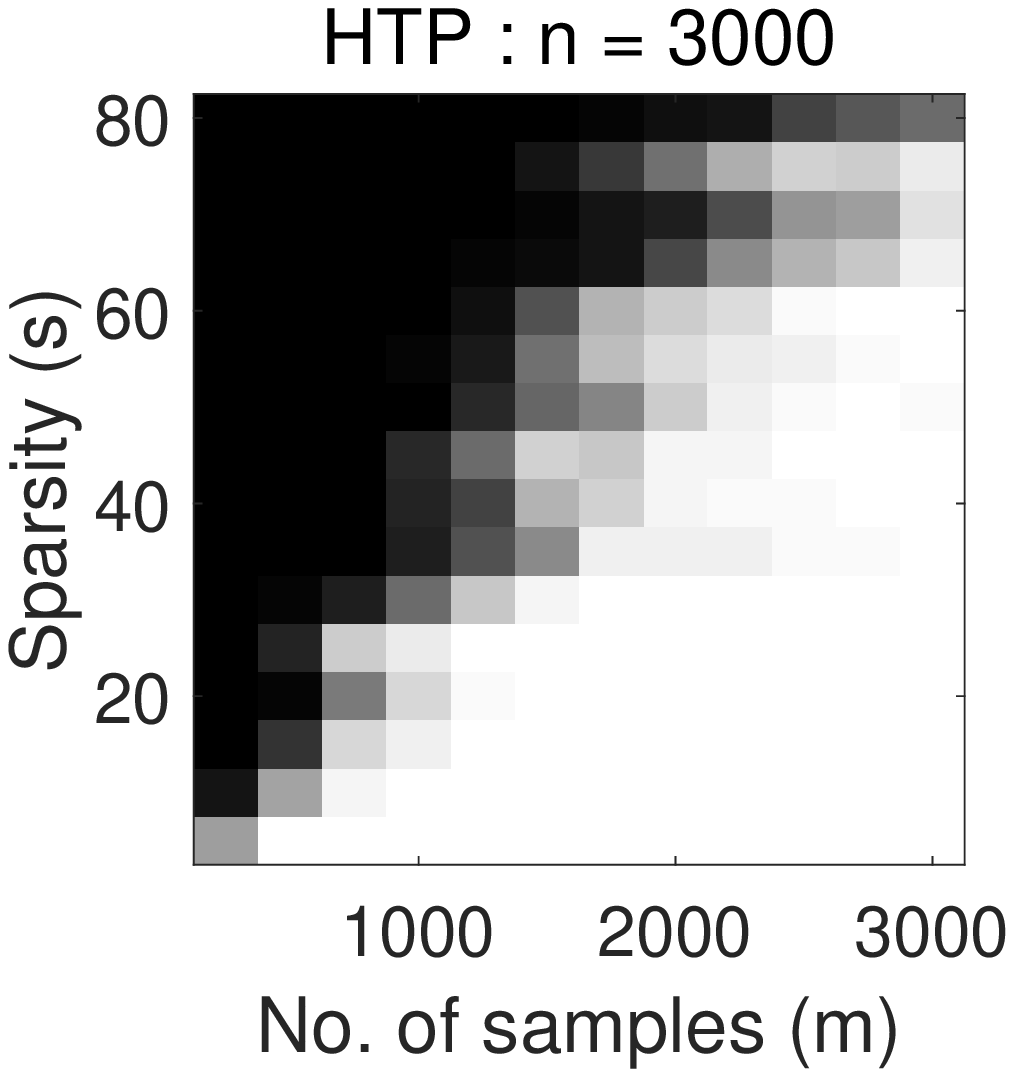}}
\caption{\label{PhaseTrans}Phase transition for different algorithms with signal dimension $n=3000$. The successful recovery rates are depicted in different grey levels of the corresponding block. Black means that the successful recovery rate is $0\%$, white $100\%$, and grey between $0\%$ and $100\%$.}
\end{figure}

\paragraph{$1$-D signal reconstruction.}
Now we recovery an one dimensional signal (in Figure~\ref{1dtruesignal}) from phaseless measurements using different methods. The sampling matrix $\bm{A}$ is of size $2800 \times 8000$ and it consists of a random Gaussian matrix and an inverse wavelet transform (with four level of Daubechies $1$ wavelet). The one-dimensional signal is sparse ($37$ nonzeros) under the wavelet transformation. The noise level in the measurements is $\sigma=0.05$. We do the reconstruction from phaseless noisy measurements by methods including HTP, CoPRAM, ThWF and SPARTA. In the numerical experiment, the exact sparsity level is assumed to be unknown and $s$ is set to be $\lfloor 0.01n\rfloor$ for reconstruction. The PSNR values is defined as $$\mathrm{PSNR}=10\cdot \log\frac{\mathrm{V}^2}{\mathrm{MSE}},$$
where $\mathrm{V}$ is the maximum absolute value of the true signal and the recovered signal, and $\mathrm{MSE}$ is the mean squared error in the reconstruction. The results are shown in Figure~\ref{1dsignalrecover}. We see that our proposed HTP algorithm is the fastest among all algorithms.
\begin{figure}[!htb]
\centering
{\includegraphics[trim =1.5cm 1cm 1.5cm 0cm,clip=true,width=0.25\textwidth]{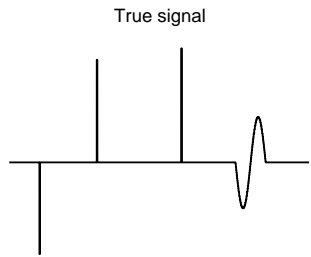}}
\caption{\label{1dtruesignal} The target one-dimensional signal.}
\end{figure}
\begin{figure}[!htb]
\centering
{\includegraphics[trim =1.5cm 1.5cm 1cm 0cm,clip=true,width=0.56\textwidth]{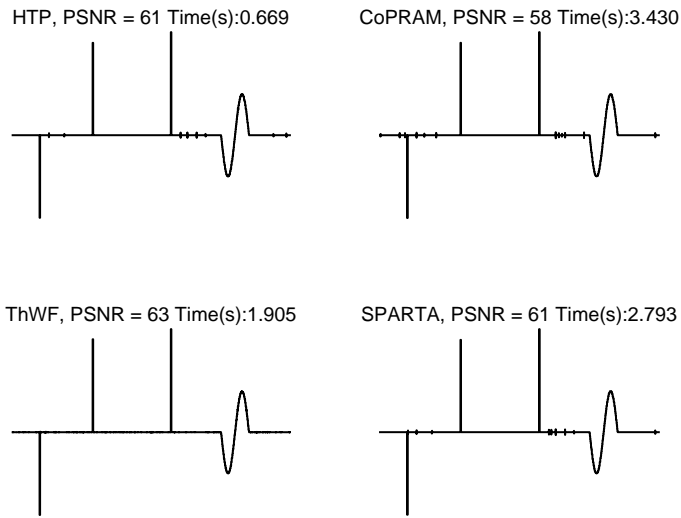}}
\caption{\label{1dsignalrecover} Reconstruction of the one-dimensional signal with $m=2800, n=8000, \sigma=0.05$. Time(s) is the running time in seconds.}
\end{figure}

\section{Proofs}\label{section:proofs}
In this section, we prove Theorem \ref{localconvergence}. To begin with, some crucial propositions/lemmas are presented in Section \ref{subsec:lemmas}. Then the proof for Parts (a) and (b) of Theorem \ref{localconvergence} are given in Sections \ref{subsec:proofa} and \ref{subsec:proofb} respectively.

We prove only the case when $\lV\bm{x}_0-\bm{x}^{\natural}\rV_2\le\lV\bm{x}_0+\bm{x}^{\natural}\rV_2$, so that $\mathrm{dist}\left(\bm{x}_0,\bm{x}^{\natural}\right)=\lV\bm{x}_0-\bm{x}^{\natural}\rV_2$. We will then estimate $\lV\bm{x}_k-\bm{x}^{\natural}\rV_2$, which is an upper bound of $\mathrm{dist}\left(\bm{x}_0,\bm{x}^{\natural}\right)$. It can be done similarly for the case when $\lV\bm{x}_0+\bm{x}^{\natural}\rV_2\le\lV\bm{x}_0-\bm{x}^{\natural}\rV_2$ by estimating $\lV\bm{x}_k+\bm{x}^{\natural}\rV_2$.

\subsection{Key Lemmas}\label{subsec:lemmas}
In this subsection, we give and prove some key propositions and lemmas that will be used in the proof of Theorem \ref{localconvergence}.

Let us first present two probabilistic propositions and a probabilistic lemma. The first probabilistic proposition is well known in compressed sensing theory \cite{foucart2013invitation,candes2005decoding}, which states that the random Gaussian matrix $\bm{A}$ satisfies the restricted isometry property (RIP) if $m$ is sufficiently large.
\begin{proposition}[{\cite[Theorem 9.27]{foucart2013invitation}}]\label{ARIP}
Let $\{\bm{a}_i\}_{i=1}^{m}$ be i.i.d. Gaussian random vectors  with mean $\bm{0}$ and variance matrix $\bm{I}$. Let $\bm{A}$ be defined in \eqref{Aandy}. There exists some universal positive constants $C_1',C_2'$ such that: For any natural number $r\leq n$ and any $\delta_r\in(0,1)$, with probability at least $1-e^{-C_1'm}$, $\bm{A}$ satisfies $r$-RIP with constant $\delta_r$, i.e.,
\begin{equation}\label{eq:RIP}
\left(1-\delta_r\right)\lV\bm{x}\rV_2^2\le\lV \bm{A}\bm{x}\rV_2^2\le \left(1+\delta_r\right)\lV\bm{x}\rV_2^2, \qquad\forall~\lV\bm{x}\rV_0\le r,
\end{equation}
provided $m\ge C_2'\delta_r^{-2}r\log\left(n/r\right)$.
\end{proposition}

With the help of RIP, we can bound the spectral norm of submatrices of $\bm{A}$. The following result is from Proposition~3.1 in \cite{needell2009cosamp}.
\begin{proposition}[{\cite[Proposition 3.1]{needell2009cosamp}}]\label{RIP}
Under the event \eqref{eq:RIP} with $r=s$ and $r=s'$, for any disjoint subsets $\sS$ and $\T$ of $\{1,2,\cdots,m\}$ satisfying $\lv \sS\rv\le s$ and $\lv \T\rv\le s'$, we have the following inequalities:
\begin{subequations}\label{neq:RIP}
\begin{align}
\lV \bm{A}^{T}_{\sS}\rV_2\le \sqrt{1+\delta_{s}}, \label{RIP:a}\\
1-\delta_{s}\le\lV \bm{A}^{T}_{\sS}\bm{A}_{\sS}\rV_2\le 1+\delta_{s}, \label{RIP:b}\\
\lV \bm{A}^{T}_{\sS}\bm{A}_{\T}\rV_2\le \delta_{s+s'}.\label{RIP:c}
\end{align}
\end{subequations}
\end{proposition}

The second probabilistic lemma is a corollary of \cite[Lemma 25]{soltanolkotabi2019structured}, and one can find same modification of the lemma in \cite[Lemma C.1.]{jagatap2019sample}.
\begin{lemma}[Corollary of {\cite[Lemma 25]{soltanolkotabi2019structured}}]\label{bound:Ax}
Let $\{\bm{a}_i\}_{i=1}^{m}$ be i.i.d. Gaussian random vectors with mean $\bm{0}$ and variance matrix $\bm{I}$. Let $\lambda_0$ be any constant in $(0,\frac{1}{8}]$. Fixing any $\varepsilon_0>0$, there exists some universal positive constants $C_3',C_4'$, if
$$
m\geq C_3's\log(n/s),
$$
then with probability at least $1-e^{-C_4' m}$ it holds that
\begin{equation}\label{event:lemma1}
\frac{1}{m}\mathop{\sum}\limits_{i=1}^{m}\lv \bm{a}_i^T\bm{x}^{\natural}\rv^2\cdot \bm{1}_{\left\{\left(\bm{a}_i^T\bm{x}\right)\left(\bm{a}_i^T\bm{x}^{\natural}\right)\le 0\right\}}
\le\frac{1}{\left(1-\lambda_0\right)^2}\left(\varepsilon_0+\lambda_0 \sqrt{\frac{21}{20}}\right)^2\lV\bm{x}-\bm{x}^{\natural}\rV_2^2,
\end{equation}
for all $\bm{x}$ satisfying $\|\bm{x}\|_0\leq s,\|\bm{x}-\bm{x}^{\natural}\|_2\leq\lambda_0\|\bm{x}^{\natural}\|_2$ .
\end{lemma}
\begin{proof}
 In fact, the left hand side of the inequality \eqref{event:lemma1} is same to the second line of \cite[eq. VIII.45]{soltanolkotabi2019structured}, and the upper bound has been given by \cite[Lemma 25]{soltanolkotabi2019structured} (see also \cite[Page 2393]{soltanolkotabi2019structured}).
\end{proof}

With those probabilistic lemmas/propositions, we can show some deterministic lemmas that are crucial to the proof of our theorem. The following lemma bound an error on $\bm{y}_{k+1}$ by the error of $\bm{x}_k$.
\begin{lemma}\label{bound:Ax-y}
Let $\left\{\bm{x}_k,\bm{y}_{k},\sS_k\right\}_{k\ge1}$ be generated by the Algorithm~\ref{alg:htp}. Assume $\lV \bm{x}_{k}-\bm{x}^{\natural}\rV_2\le \lambda_0 \lV \bm{x}^{\natural}\rV_2$. Then under the event \eqref{eq:RIP} with $r=s,2s$ and the event \eqref{event:lemma1}, it holds that
\begin{align*}
\lV \bm{A}^{T}_{\sS_{k+1}}\left(\bm{y}_{k+1}-\bm{A}\bm{x}^{\natural}\right)\rV_2
\le \sqrt{C_{\lambda_0}\left(1+\delta_s\right)}\lV \bm{x}_{k}-\bm{x}^{\natural}\rV_2,
\end{align*}
where $C_{\lambda_0}=\frac{4}{(1-\lambda_0)^2}\left(\varepsilon_0+\lambda_0 \sqrt{\frac{21}{20}}\right)^2$, $\varepsilon_0=10^{-3}$.
\end{lemma}
\begin{proof}

Let the sets $\{\G_k\}_{k\ge1}$ defined to be
$$\G_k=\{i:\mathrm{sgn}\left(\bm{a}_i^T\bm{x}_k\right)=\mathrm{sgn}\left(\bm{a}_i^T\bm{x}^{\natural}\right),\ 1
\le i\le m \},\ k=1,2,3,\cdots.$$
Recall that $\bm{y}_{k+1}:=\bm{y} \odot \mathrm{sgn}{\left(\bm{\bm{A}}\bm{x}_{k}\right)}$. We then have
\begin{align}\label{bound:yk-Ax}
\begin{split}
\lV \bm{y}_{k+1}-\bm{A}\bm{x}^{\natural}\rV_2^2
&=\frac{1}{m}\mathop{\sum}\limits_{i=1}^{m}\left(\mathrm{sgn}\left(\bm{a}_i^T\bm{x}_k\right)-\mathrm{sgn}\left(\bm{a}_i^T\bm{x}^{\natural}\right)\right)^2
\lvert \bm{a}_i^T\bm{x}^{\natural}\lvert^2 \\
&\le  \frac{4}{m}\mathop{\sum}\limits_{i\in \G^c_k} \lv \bm{a}_i^T\bm{x}^{\natural}\rv^2 \cdot \bm{1}_{\left\{\left(\bm{a}_i^T\bm{x}_k\right)\left(\bm{a}_i^T\bm{x}^{\natural}\right)\le 0\right\}}\\
&\le\underbrace{\frac{4}{(1-\lambda_0)^2}\left(\varepsilon_0+\lambda_0 \sqrt{\frac{21}{20}}\right)^2}_{C_{\lambda_0}}\lV \bm{x}_k-\bm{x}^{\natural}\rV_2^2.
\end{split}
\end{align}
where the second line follows from $\lvert\mathrm{sgn}\left(\bm{a}_i^T\bm{x}\right)-\mathrm{sgn}\left(\bm{a}_i^T\bm{x}^{\natural}\right)\lvert\le 2$ and $\mathrm{sgn}\left(\bm{a}_i^T\bm{x}_k\right)-\mathrm{sgn}\left(\bm{a}_i^T\bm{x}^{\natural}\right)=0$ on $\G_k$, the last line follows from Lemma~\ref{bound:Ax} with a fixed $\varepsilon_0=10^{-3}$. Together with \eqref{RIP:a} in Proposition \ref{RIP}, \eqref{bound:yk-Ax} leads to
\begin{align*}
\lV \bm{A}^{T}_{\sS_{k+1}}\left(\bm{y}_{k+1}-\bm{A}\bm{x}^{\natural}\right)\rV_2
\le \sqrt{C_{\lambda_0}\left(1+\delta_s\right)}\lV \bm{x}_{k}-\bm{x}^{\natural}\rV_2.
\end{align*}
\end{proof}
The last key lemma estimate the error of the vector obtained by one iteration of IHT. Its proof uses a similar strategy to the proof of  in \cite[Lemma~3]{wang2016sparse}. To make the paper self-contained, we have included the details of the proof.
\begin{lemma}\label{contraction:u0}
Let $\left\{\bm{x}_k,\bm{y}^{k},\sS_k\right\}_{k\ge1}$ be the sequence generated by Algorithm~\ref{alg:htp}. Define
$$
\bm{u}_{k+1}:=\H_s\big(\bm{x}_{k}+\mu \bm{A}^T\left( \bm{y} _{k+1}-\bm{A}\bm{x}_{k}\right)\big).
$$
Assume $\lV \bm{x}_{k}-\bm{x}^{\natural}\rV_2\le \lambda_0\lV \bm{x}^{\natural}\rV_2$.
Under the event \eqref{eq:RIP} with $r=s,2s,3s$ and the event \eqref{event:lemma1}, it holds that
$$
\lV \bm{u}^{k+1}-\bm{x}^{\natural}\rV_2\le \rho\lV \bm{x}_{k}-\bm{x}^{\natural}\rV_2,
$$
where $\rho=2\left(\sqrt{2}\max\{\mu\delta_{3s},1-\mu\left(1-\delta_{2s}\right)\}+\mu\sqrt{C_{\lambda_0}\left(1+\delta_{2s}\right)}\right)$ with $\mu<\frac{1}{1+\delta_{2s}}$.
\end{lemma}
\begin{proof}
Define $\sS_{\natural}:=\mathrm{supp}\left(\bm{x}^{\natural}\right)$, $\T_{k+1}:=\sS_{k+1}\bigcup \sS_{\natural}$, and
$$
\bm{v}_{k+1}:=\bm{x}_{k}+\mu \bm{A}^T\left( \bm{y}_{k+1}-\bm{A}\bm{x}_{k}\right).
$$
Since $ \bm{u}_{k+1}$ is the best $s$-term approximation of $\bm{v}_{k+1}$, we have
$$
\lV\bm{u}_{k+1}-\bm{v}_{k+1}\rV_2\le \lV\bm{x}^{\natural}-\bm{v}_{k+1}\rV_2,
$$
which together with $\mathrm{supp}\left(\bm{u}_{k+1}\right)\subseteq \T_{k+1}$ and $\mathrm{supp}\left(\bm{x}^{\natural}\right)\subseteq \T_{k+1}$ implies
$$
\lV[\bm{u}_{k+1}]_{\T_{k+1}}-[\bm{v}_{k+1}]_{\T_{k+1}}\rV_2\le \lV[\bm{x}^{\natural}]_{\T_{k+1}}-[\bm{v}_{k+1}]_{\T_{k+1}}\rV_2.
$$
Then, by the triangle inequality and the inequality above, we obtain
\begin{align}\label{bound:ud-xd}
\begin{split}
\lV [\bm{u}_{k+1}]_{\T_{k+1}}-[\bm{x}^{\natural}]_{\T_{k+1}}\rV_2
&=\lV [\bm{u}_{k+1}]_{\T_{k+1}}-[\bm{v}_{k+1}]_{\T_{k+1}}+[\bm{v}_{k+1}]_{\T_{k+1}}-[\bm{x}^{\natural}]_{\T_{k+1}}\rV_2\\
&\le \lV [\bm{u}_{k+1}]_{\T_{k+1}}-[\bm{v}_{k+1}]_{\T_{k+1}}\rV_2+\lV[\bm{v}_{k+1}]_{\T_{k+1}}-[\bm{x}^{\natural}]_{\T_{k+1}}\rV_2\\
&\le 2\lV [\bm{x}^{\natural}]_{\T_{k+1}}-[\bm{v}_{k+1}]_{\T_{k+1}}\rV_2.
\end{split}
\end{align}
Using definition of $\bm{v}_{k+1}$, a direct calculation gives
\begin{align}\label{bound:vd-xd}
\begin{split}
&\lV [\bm{v}_{k+1}]_{\T_{k+1}}-[\bm{x}^{\natural}]_{\T_{k+1}}\rV_2\\
=&\lV[\bm{x}_{k}]_{\T_{k+1}}-[\bm{x}^{\natural}]_{\T_{k+1}}-\mu \bm{A}^{T}_{\T_{k+1}}\bm{A}(\bm{x}_{k}-\bm{x}^{\natural})
+\mu \bm{A}^{T}_{\T_{k+1}}\big(\bm{y}_{k+1}-\bm{A}\bm{x}^{\natural}\big)\rV_2\\
\le& \underbrace{\lV\left(\bm{I}-\mu \bm{A}^{T}_{\T_{k+1}}\bm{A}_{\T_{k+1}}\right)\left([\bm{x}_{k}]_{\T_{k+1}}-[\bm{x}^{\natural}]_{\T_{k+1}}\right)\rV_2}_{I_1}
+\underbrace{\lV\mu \bm{A}^{T}_{\T_{k+1}}\bm{A}_{\T_{k}\backslash\T_{k+1}}\left[\bm{x}_{k}-\bm{x}^{\natural}\right]_{\T_{k}\backslash\T_{k+1}}\rV_2}_{I_2} \\
&\qquad+\underbrace{\lV\mu \bm{A}^{T}_{\T_{k+1}}\left(\bm{y}_{k+1}-\bm{A}\bm{x}^{\natural}\right)\rV_2}_{I_3}.
\end{split}
\end{align}
Let us estimate $I_1$, $I_2$, and $I_3$ one by one.
\begin{itemize}
\item For $I_1$: It follows from \eqref{RIP:b} in Proposition \ref{RIP}, $\mu\in\left(0,\frac{1}{1+\delta_{2s}}\right)$ and Weyl's inequality that
$$
1-\mu\left(1+\delta_{2s}\right)\le\lV\bm{I}-\mu \bm{A}^{T}_{\T_{k+1}}\bm{A}_{\T_{k+1}}\rV_2\le
 1-\mu\left(1-\delta_{2s}\right),
$$
which implies
$$
I_1\leq \left(1-\mu\left(1-\delta_{2s}\right)\right)\lV[\bm{x}_{k}]_{\T_{k+1}}-[\bm{x}^{\natural}]_{\T_{k+1}}\rV_2.
$$

\item For $I_2$: Eq. \eqref{RIP:c} in Proposition \ref{RIP} implies
$$
I_2
\le \mu\delta_{3s}\lV[\bm{x}_{k}-\bm{x}^{\natural}]_{\T_{k}\setminus\T_{k+1}}\rV_2.
$$

\item For $I_3$: Lemma \ref{bound:Ax-y} gives directly
\begin{align}\label{term2:vd-xd}
\lV\mu \bm{A}^{T}_{\T_{k+1}}\left(\bm{y}_{k+1}-\bm{A}\bm{x}^{\natural}\right)\rV_2
\le \mu\sqrt{C_{\lambda_0}\left(1+\delta_{2s}\right)}\lV\bm{x}_{k}-\bm{x}^{\natural}\rV_2
\end{align}
\end{itemize}

Combining all terms together, we obtain
\begin{equation}
\begin{split}
\lV [\bm{v}_{k+1}]_{\T_{k+1}}-[\bm{x}^{\natural}]_{\T_{k+1}}\rV_2
&\leq I_1+I_2+I_3\leq \sqrt{2(I_1^2+I_2^2)}+I_3\cr
&\leq \sqrt{2}\max\{\mu\delta_{3s},1-\mu\left(1-\delta_{2s}\right)\}\lV\bm{x}_{k}-\bm{x}^{\natural}\rV_2+\mu\sqrt{C_{\lambda_0}\left(1+\delta_{2s}\right)}\lV\bm{x}_{k}-\bm{x}^{\natural}\rV_2\cr
&=\left(\sqrt{2}\max\{\mu\delta_{3s},1-\mu\left(1-\delta_{2s}\right)\}+\mu\sqrt{C_{\lambda_0}\left(1+\delta_{2s}\right)}\right)\lV\bm{x}_{k}-\bm{x}^{\natural}\rV_2.
\end{split}
\end{equation}
We conclude the proof by using \eqref{bound:ud-xd}.
\end{proof}

\subsection{Proof of Part (a) of Theorem \ref{localconvergence}}
\label{subsec:proofa}
Now we are ready to prove Part (a) of Theorem \ref{localconvergence}, i.e., the local convergence with a linear rate.
\begin{proof}[Proof of Part (a) of Theorem \ref{localconvergence}]
Under the event \eqref{eq:RIP} with $r=s,2s,3s$ and the event \eqref{event:lemma1}, the theorem is proved by induction. Suppose $\lV \bm{x}_{k}-\bm{x}^{\natural}\rV_2\leq\lambda_0\lV \bm{x}^{\natural}\rV_2$. Define $\sS_{\natural}=\mathrm{supp}\left(\bm{x}^{\natural}\right)$. The optimality condition \eqref{eq:optxk+1} gives
\begin{align*}
\bm{A}^{T}_{\sS_{k+1}}\bm{A}_{\sS_{k+1}}\left( [\bm{x}_{k+1}]_{\sS_{k+1}}-[\bm{x}^{\natural}]_{\sS_{k+1}}\right)
&=\bm{A}^{T}_{\sS_{k+1}}\left( \bm{y}_{k+1}-\bm{A}_{\sS_{k+1}}[\bm{x}^{\natural}]_{\sS_{k+1}}\right)\\
&=\bm{A}^{T}_{\sS_{k+1}}\left(\bm{y}_{k+1}-\bm{A}\bm{x}^{\natural}\right)+\bm{A}^{T}_{\sS_{k+1}}\bm{A}_{\sS_{k+1}^c}[\bm{x}^{\natural}]_{\sS_{k+1}^c}\\
&=\bm{A}^{T}_{\sS_{k+1}}\left(\bm{y}_{k+1}-\bm{A}\bm{x}^{\natural}\right) +\bm{A}^{T}_{\sS_{k+1}}\bm{A}_{\sS_{\natural}\setminus\sS_{k+1}}[\bm{x}^{\natural}]_{\sS_{\natural}\setminus\sS_{k+1}}.
\end{align*}
In view of Lemma~\ref{bound:Ax-y} and Proposition~\ref{RIP}, this leads to
\begin{equation}\label{eq:est1inproofa}
\begin{split}
\left(1-\delta_s\right)\lV [\bm{x}_{k+1}]_{\sS_{k+1}}-[\bm{x}^{\natural}]_{\sS_{k+1}}\rV_2 &\le\lV \bm{A}^{T}_{\sS_{k+1}}\bm{A}_{\sS_{k+1}}\left( [\bm{x}_{k+1}]_{\sS_{k+1}}-[\bm{x}^{\natural}]_{\sS_{k+1}}\right)\rV_2\\
&\le\lV\bm{A}^{T}_{\sS_{k+1}}\left(\bm{y}_{k+1}-\bm{A}\bm{x}^{\natural}\right)\rV_2 +\lV\bm{A}^{T}_{\sS_{k+1}}\bm{A}_{\sS_{\natural}\setminus\sS_{k+1}}[\bm{x}^{\natural}]_{\sS_{\natural}\setminus\sS_{k+1}}\rV_2\\
&\le \sqrt{C_{\lambda_0}\left(1+\delta_s\right)}\lV \bm{x}_{k}-\bm{x}^{\natural}\rV_2+\delta_{2s}\lV[\bm{x}^{\natural}]_{\sS_{\natural}\setminus\sS_{k+1}} \rV_2.
\end{split}
\end{equation}
Moreover, since $[\bm{x}^{\natural}]_{\sS_{\natural}\setminus\sS_{k+1}}$ is a subvector of $\bm{x}^{\natural}-\bm{u}_{k+1}$, Lemma~\ref{contraction:u0} implies
\begin{equation}\label{eq:est2inproofa}
\lV [\bm{x}^{\natural}]_{\sS_{\natural}\setminus\sS_{k+1}} \rV_2\le \lV \bm{u}_{k+1}-\bm{x}^{\natural}\rV_2\le \rho\lV \bm{x}_{k}-\bm{x}^{\natural}\rV_2,
\end{equation}
where
$\rho=2\left(\sqrt{2}\max\{\mu\delta_{3s},1-\mu\left(1-\delta_{2s}\right)\}+\mu\sqrt{C_{\lambda_0}\left(1+\delta_{2s}\right)}\right)$ with $\sqrt{C_{\lambda_0}}=\frac{2}{(1-\lambda_0)}\left(\varepsilon_0+\lambda_0 \sqrt{\frac{21}{20}}\right)$, $\varepsilon_0=10^{-3}$. Eq. \eqref{eq:est2inproofa} is plugged into \eqref{eq:est1inproofa} to yield
\begin{equation}\label{eq:erronSk+1}
\lV [\bm{x}_{k+1}]_{\sS_{k+1}}-[\bm{x}^{\natural}]_{\sS_{k+1}}\rV_2
\le \frac{\sqrt{C_{\lambda_0}\left(1+\delta_s\right)}+\delta_{2s}\rho}{1-\delta_s}\lV \bm{x}_{k}-\bm{x}^{\natural}\rV_2.
\end{equation}
For the error on $\sS_{k+1}^c$, we use \eqref{eq:est2inproofa} again as follows
\begin{equation}\label{eq:erronSk+1c}
\lV [\bm{x}_{k+1}]_{\sS_{k+1}^c}-[\bm{x}^{\natural}]_{\sS_{k+1}^c}\rV_2=\lV \bm{x}^{\natural}_{S_{k+1}^c} \rV_2=\lV [\bm{x}^{\natural}]_{\sS_{\natural}\setminus\sS_{k+1}} \rV_2\le \rho\lV \bm{x}_{k}-\bm{x}^{\natural}\rV_2.
\end{equation}
Combining \eqref{eq:erronSk+1} and \eqref{eq:erronSk+1c}, we obtain
\begin{align*}
\lV\bm{x}_{k+1}-\bm{x}^{\natural}\rV_2^2
\le\underbrace{\left( \left(\frac{\sqrt{C_{\lambda_0}\left(1+\delta_s\right)}+\delta_{2s}\rho}{1-\delta_s}\right)^2+\rho^2\right)}_{\alpha^2}\lV \bm{x}_{k}-\bm{x}^{\natural}\rV_2^2,
\end{align*}
Since $\delta_s\le\delta_{2s}\le\delta_{3s}$, $\rho$ can be small as $\delta_{3s}, \lambda_0$ approach to $0$. Moreover, small $\lambda_0$ give a small $C_{\lambda_0}$ (since $\varepsilon_0$ is small). Therefore, one can set proper parameters $\delta_{3s}$, $\lambda_0$ to make $\alpha<1$. For example, recall $\varepsilon_0=10^{-3}$, then for $\delta_{3s}\le 0.05$ and $\mu=0.95$, we have $\alpha\in (0,1)$ if provided $\lambda_0\le \frac{1}{8}$.  Since $\lV\bm{x}_{k+1}-\bm{x}^{\natural}\rV_2\leq \alpha\lV\bm{x}_{k}-\bm{x}^{\natural}\rV_2\leq \lambda_0 \lV\bm{x}^{\natural}\rV_2$, the hypothesis of the induction is satisfied. Therefore, by induction,
$$
\lV\bm{x}_{k+1}-\bm{x}^{\natural}\rV_2\leq \alpha\lV\bm{x}_{k}-\bm{x}^{\natural}\rV_2,\quad\forall~k\geq 0.
$$
\end{proof}

\subsection{Proof of Part (b) of Theorem \ref{localconvergence}}\label{subsec:proofb}
In the following, we prove Part (b) of Theorem \ref{localconvergence}, i.e., the finite-step termination of Algorithm \ref{alg:htp}.

\begin{proof}[Proof of Part (b) of Theorem \ref{localconvergence}]
This part is proved under the event that Part (a) holds.

Let $k_1$ be the minimum integer that satisfies
\begin{equation}\label{eq:k1}
\lambda_0\|\bm{x}^{\natural}\|_2\alpha^{k_1}< |x_{\min}^{\natural}|,
\end{equation}
where $x_{\min}^{\natural}$ is the smallest nonzero entry of $\bm{x}^{\natural}$ in magnitude. Then we must have $\sS_{\natural}\subseteq\sS_k$ for all $k\geq k_1$, because otherwise Part (a) of Theorem \ref{localconvergence} implies $
\|\bm{x}_k-\bm{x}^{\natural}\|_2\leq \lambda_0\|\bm{x}^{\natural}\|_2\alpha^{k}\leq \lambda_0\|\bm{x}^{\natural}\|_2\alpha^{k_1}< |x_{\min}^{\natural}|
$,
which contradicts with $\|\bm{x}_k-\bm{x}^{\natural}\|_2\geq|x_i^{\natural}|\geq|x_{\min}^{\natural}|$ for some $i\in\sS_{\natural}\setminus\sS_{k}\neq\emptyset$.

Now we consider $k\geq k_1$. Let $y_{\min}$ be the minimum nonzero entry of $\bm{y}$.  Since $\sS_{\natural}\subseteq\sS_{k+1}$,
\begin{align*}
\lv\l \bm{A}^{T}_{\sS_{k+1}}\left(\bm{y}_{k+1}-\bm{A}\bm{x}^{\natural}\right),[\bm{x}^{\natural}]_{\sS_{k+1}}\r\rv
&=\lv\l \bm{A}^{T}_{\sS_{\natural}}\left(\bm{y}_{k+1}-\bm{A}\bm{x}^{\natural}\right),[\bm{x}^{\natural}]_{S_{\natural}}\r\rv=\lv\l\bm{y}_{k+1}-\bm{A}\bm{x}^{\natural},\bm{A}_{\sS_{\natural}}[\bm{x}^{\natural}]_{\sS_{\natural}}\r\rv\\
&=\lv\l |\bm{A}\bm{x}^{\natural}| \odot \mathrm{sgn}\left(\bm{A}\bm{x}_{k}\right)-\bm{A}\bm{x}^{\natural},\bm{A}\bm{x}^{\natural} \r\rv=\sum\limits_{i\in \G_{k}^c}\frac{2}{m} |\bm{a}_i^T\bm{x}^{\natural}|^2\\
&\ge 2\lv\G_{k}^c \rv |y_{\min}|^2,
\end{align*}
where $\G_k=\{i:\mathrm{sgn}\left(\bm{a}_i^T\bm{x}_k\right)=\mathrm{sgn}\left(\bm{a}_i^T\bm{x}^{\natural}\right)\}$.  Thus,
\begin{align}\label{bound:Gc}
\begin{split}
\lv\G_{k}^c \rv
&\le \frac{1}{2|y_{\min}|^2}\lv\l \bm{A}^{T}_{\sS_{k+1}}\left(\bm{y}_{k+1}-\bm{A}\bm{x}^{\natural}\right),[\bm{x}^{\natural}]_{\sS_{k+1}}\r\rv
\leq \frac{1}{2|y_{\min}|^2}\lV\bm{A}^{T}_{\sS_{k+1}}\left(\bm{y}_{k+1}-\bm{A}\bm{x}^{\natural}\right)\rV_2\lV[\bm{x}^{\natural}]_{\sS_{k+1}}\rV_2\\
&\leq\frac{1}{2|y_{\min}|^2}\sqrt{C_{\lambda_0}\left(1+\delta_s\right)}\lV \bm{x}_{k}-\bm{x}^{\natural}\rV_2
\lV\bm{x}^{\natural}\rV_2
\leq \frac{\lambda_0\sqrt{C_{\lambda_0}\left(1+\delta_{s}\right)}\lV\bm{x}^{\natural}\rV_2^2 }{2|y_{\min}|^2}\alpha^{k},
\end{split}
\end{align}
where the last three inequalities follow from Cauchy-Schwartz inequality, Lemma \ref{bound:Ax-y}, and Part (a) of Theorem \ref{localconvergence} respectively. Define $k_2$ be the minimum integer such that
$$
\frac{\lambda_0\sqrt{C_{\lambda_0}\left(1+\delta_{s}\right)}\lV\bm{x}^{\natural}\rV_2^2 }{2|y_{\min}|^2}\alpha^{k_2}<1.
$$
Then, for all $k\geq\max\{k_1,k_2\}$, we have $\lv\G_{k}^c \rv<1$. Since $\lv\G_{k}^c \rv$ is an integer, $\lv\G_{k}^c \rv=0$ for all $k\geq\max\{k_1,k_2\}$, which implies $\bm{y}_{k+1}=\bm{A}\bm{x}^{\natural}$ for all $k\geq\max\{k_1,k_2\}$.

Now we consider all $k$ satisfying $k\geq\max\{k_1,k_2\}$, so that $\sS_{\natural}\subset\sS_{k+1}$ and $\bm{y}_{k+1}=\bm{A}\bm{x}^{\natural}$. Then we have $\bm{x}_{k+1}=\mathop{\mathrm{arg~min}}_{\mathrm{supp}(\bm{x})\subset\sS_{k+1}}\|\bm{A}\bm{x}-\bm{y}_{k+1}\|_2=\mathop{\mathrm{arg~min}}_{\mathrm{supp}(\bm{x})\subset\sS_{k+1}}\|\bm{A}\bm{x}-\bm{A}\bm{x}^{\natural}\|_2$. Since $|\sS_{k+1}|\leq s$ and $\sS_{\natural}\subset\sS_{k+1}$, the zeroth order optimality condition and the RIP imply
$$
\sqrt{1-\delta_{2s}}\|\bm{x}_{k+1}-\bm{x}^{\natural}\|_2\leq\|\bm{A}\bm{x}_{k+1}-\bm{A}\bm{x}^{\natural}\|_2\leq \|\bm{A}\bm{x}^{\natural}-\bm{A}\bm{x}^{\natural}\|_2=0.
$$
So we have $\bm{x}_{k+1}=\bm{x}^{\natural}$ for all $k\geq\max\{k_1,k_2\}$.

It remains to estimate $k_1$ and $k_2$.
\begin{itemize}
\item For $k_1$: The lower bound of $k_1$ is obtained straightforwardly from \eqref{eq:k1} as $k_1>\frac{\log(\lambda_0\|\bm{x}^{\natural}\|_2/|x_{\min}^{\natural}|)}{\log(\alpha^{-1})}$. Therefore,
$$
k_1=\left\lfloor\frac{\log(\lambda_0\|\bm{x}^{\natural}\|_2/|x_{\min}^{\natural}|)}{\log(\alpha^{-1})}\right\rfloor+1
\leq C_2\log\frac{\|\bm{x}^{\natural}\|_2}{|x_{\min}^{\natural}|}+C_3,
$$
where $\lfloor\cdot\rfloor$ is the floor operation.
\item For $k_2$: Obviously,
$$
k_2=\left\lfloor\frac{\log\left(\frac12\lambda_0\sqrt{C_{\lambda_0}\left(1+\delta_{s}\right)}\lV\bm{x}^{\natural}\rV_2^2/|y_{\min}|^2\right)}{\log(\alpha^{-1})}\right\rfloor+1.
$$
Therefore, to upper bound $k_2$, it suffices to lower bound $|y_{\min}|$. Since $\{\bm{a}_i\}_{i=1}^m$ are independent random Gaussian vectors and $\bm{x}^{\natural}$ is fixed, $\{\bm{a}_i^T\bm{x}^{\natural}\}_{i=1}^m$ are independent random Gaussian variable with mean $0$ and variance $\|\bm{x}^{\natural}\|_2^2$. Let $\epsilon>0$ be a fixed constant. Then, for any $i=1,\ldots,m$,
\begin{equation*}
\begin{split}
\mathrm{Prob}\left\{|\bm{a}_i^T\bm{x}^{\natural}|\geq\epsilon\right\}&=\frac{2}{\|\bm{x}^{\natural}\|_2\sqrt{2\pi}}\int_{\epsilon}^{+\infty}e^{-\frac{t^2}{2\|\bm{x}^{\natural}\|_2^2}}dt
=1-\frac{2}{\|\bm{x}^{\natural}\|\sqrt{2\pi}}\int_{0}^{\epsilon}e^{-\frac{t^2}{2\|\bm{x}^{\natural}\|_2^2}}dt\cr
&\geq 1-\frac{2}{\|\bm{x}^{\natural}\|\sqrt{2\pi}}\cdot e^{-\frac{0}{2\|\bm{x}^{\natural}\|_2^2}}\cdot\epsilon
= 1-\sqrt{\frac{2}{\pi}}\frac{\epsilon}{\|\bm{x}^{\natural}\|_2}.
\end{split}
\end{equation*}
Due to the independency of $\{\bm{a}_i^T\bm{x}^{\natural}\}_{i=1}^m$, we obtain
$$
\mathrm{Prob}\left\{|\bm{a}_i^T\bm{x}^{\natural}|\geq\epsilon,\quad\forall~i=1,\ldots,m\right\}\geq \left(1-\sqrt{\frac{2}{\pi}}\frac{\epsilon}{\|\bm{x}^{\natural}\|_2}\right)^m
\geq 1-\sqrt{\frac{2}{\pi}}\frac{m\epsilon}{\|\bm{x}^{\natural}\|_2}.
$$
Therefore, if we choose $\epsilon=m^{-\beta}\|\bm{x}^{\natural}\|_2\sqrt{\frac{\pi}{2}}$, then with probability at least $1-m^{1-\beta}$ we have $y_{\min}=\min_{1\leq i\leq m}\frac{1}{\sqrt{m}}|\bm{a}_i^T\bm{x}^{\natural}|\geq m^{-\beta-\frac{1}{2}}\|\bm{x}^{\natural}\|_2\sqrt{\frac{\pi}{2}}$, and thus
$$
k_2\leq\left\lfloor\frac{\log\left(\frac{1}{\pi}m^{2\beta+1}\lambda_0\sqrt{C_{\lambda_0}\left(1+\delta_{s}\right)}\right)}{\log(\alpha^{-1})}\right\rfloor+1
\leq C_2\beta\log m+C_3.
$$
\end{itemize}
\end{proof}
\subsection{Proof of Corollary \ref{localconvergence:noisy}}\label{subsec:proofnoisy}
\begin{proof}[Proof of Corollary \ref{localconvergence:noisy}]
In the noisy case, $\bm{y}=\bm{y}^{(\varepsilon)}:=|\bm{A}\bm{x}^{\natural}|+\bm{\varepsilon}$ for some $\bm{\varepsilon}\in \mathbb{R}^m$. The proof is done under the event \eqref{eq:RIP} with $r=s,2s,3s$ and the event \eqref{event:lemma1}.  Again, we let the sequence $\{\bm{x}_k,\bm{y}_k\}_{k\ge1 }$ be generated by HTP from the noisy data. In this case,  $\{\bm{y}_k\}_{k\ge1 }$ is given by
 $$\bm{y}_{k+1}=\bm{y}^{(\varepsilon)} \odot \mathrm{sgn}{\left(\bm{\bm{A}}\bm{x}_{k}\right)}=\left(|\bm{A}\bm{x}^{\natural}|+\bm{\varepsilon}\right) \odot \mathrm{sgn}{\left(\bm{\bm{A}}\bm{x}_{k}\right)}.$$
Then, using the same argument to the proof of the inequality in Lemma~\ref{bound:Ax-y}, we have
\begin{align}\label{noisy:Ax-y}
\begin{split}
\lV \bm{A}^{T}_{\sS_{k+1}}\left(\bm{y}_{k+1}-\bm{A}\bm{x}^{\natural}\right)\rV_2&=\lV \bm{A}^{T}_{\sS_{k+1}}\left(\left(|\bm{A}\bm{x}^{\natural}|+\bm{\varepsilon}\right) \odot \mathrm{sgn}{\left(\bm{\bm{A}}\bm{x}_{k}\right)} -\bm{A}\bm{x}^{\natural}\right)\rV_2\\
&\le \lV \bm{A}^{T}_{\sS_{k+1}}\left(|\bm{A}\bm{x}^{\natural}| \odot \mathrm{sgn}{\left(\bm{\bm{A}}\bm{x}_{k}\right)} -\bm{A}\bm{x}^{\natural}\right)\rV_2+\lV \bm{A}^{T}_{\sS_{k+1}}\left(\bm{\varepsilon} \odot \mathrm{sgn}{\left(\bm{\bm{A}}\bm{x}_{k}\right)}\right) \rV_2\\
&\le \sqrt{C_{\lambda_0}\left(1+\delta_s\right)}\lV \bm{x}_{k}-\bm{x}^{\natural}\rV_2+\sqrt{1+\delta_s}\lV \bm{\varepsilon}\rV_2.
\end{split}
\end{align}
where the last inequality follows from Lemma~\ref{bound:Ax-y} and \eqref{RIP:a} in Proposition \ref{RIP}.

Then, we modify Lemma~\ref{contraction:u0} to the noisy case. All arguments in Lemma~\ref{contraction:u0} go through except that the estimation of $I_3$ in \eqref{term2:vd-xd} should be replaced by \eqref{noisy:Ax-y}. Thus we obtain
\begin{equation}\label{eq:contraction111}
\lV \bm{u}_{k+1}-\bm{x}^{\natural}\rV_2\le \rho\lV \bm{x}_{k}-\bm{x}^{\natural}\rV_2+2\sqrt{1+\delta_s}\lV \bm{\varepsilon}\rV_2.
\end{equation}
where $\bm{u}_{k+1}, \rho$ are the same as those in Lemma~\ref{contraction:u0}.

The rest of the proof is a modification the proof of Part(a) of Theorem \ref{localconvergence}. Notice that the proof of Part (a) of Theorem \ref{localconvergence}  keeps unchanged until \eqref{eq:erronSk+1c}. Since $[\bm{x}^{\natural}]_{\sS_{\natural}\setminus\sS_{k+1}}$ is a subvector of $\bm{x}^{\natural}-\bm{u}_{k+1}$, by applying \eqref{eq:contraction111}, inequality \eqref{eq:erronSk+1c} in the noisy case becomes
\begin{equation*}
\lV [\bm{x}_{k+1}]_{\sS_{k+1}^c}-[\bm{x}^{\natural}]_{\sS_{k+1}^c}\rV_2=\lV \bm{x}^{\natural}_{S_{k+1}^c} \rV_2\le \lV [\bm{x}^{\natural}]_{\sS_{\natural}\setminus\sS_{k+1}} \rV_2\le \rho\lV \bm{x}_{k}-\bm{x}^{\natural}\rV_2+2\sqrt{1+\delta_s}\lV \bm{\varepsilon}\rV_2.
\end{equation*}
Moreover, with the help of \eqref{eq:est1inproofa} and \eqref{noisy:Ax-y}, inequality \eqref{eq:erronSk+1} is revised as
\begin{align*}
\lV [\bm{x}_{k+1}]_{\sS_{k+1}}-[\bm{x}^{\natural}]_{\sS_{k+1}}\rV_2
&\le \frac{1}{1-\delta_s}\lV\bm{A}^{T}_{\sS_{k+1}}\left(\bm{y}_{k+1}-\bm{A}\bm{x}^{\natural}\right)\rV_2 +\frac{\delta_{2s}}{1-\delta_s}\lV[\bm{x}^{\natural}]_{\sS_{\natural}\setminus\sS_{k+1}} \rV_2\\
&\le \frac{\sqrt{C_{\lambda_0}\left(1+\delta_s\right)}+\delta_{2s}\rho}{1-\delta_s}\lV \bm{x}_{k}-\bm{x}^{\natural}\rV_2+\frac{(1+2\delta_{2s})\sqrt{1+\delta_s}}{1-\delta_s}\lV \bm{\varepsilon}\rV_2.
\end{align*}
Putting the above two inequalities together and by $\sqrt{a^2+b^2}\le |a|+|b|$, we then have
\begin{align*}
\lV \bm{x}_{k+1}-\bm{x}^{\natural}\rV_2&= \sqrt{\lV [\bm{x}_{k+1}]_{\sS_{k+1}^c}-[\bm{x}^{\natural}]_{\sS_{k+1}^c}\rV_2^2+\lV [\bm{x}_{k+1}]_{\sS_{k+1}}-[\bm{x}^{\natural}]_{\sS_{k+1}}\rV_2^2}\\
&\le \underbrace{\left(\rho+\frac{\sqrt{C_{\lambda_0}\left(1+\delta_s\right)}+\delta_{2s}\rho}{1-\delta_s}\right)}_{\alpha_1}\lV \bm{x}_{k}-\bm{x}^{\natural}\rV_2+\underbrace{\left(\frac{1+2\delta_{2s}}{1-\delta_s}+2\right)\sqrt{1+\delta_s}}_{d}\lV \bm{\varepsilon}\rV_2.
\end{align*}
Therefore, as long as $\delta_{3s}$, $\lambda_0$ is sufficiently small, we can have $\alpha_1<1$. For example, for the fixed $\varepsilon_0=10^{-3}$, let $\delta_{3s}\le 0.05$, and we set $\mu=0.95$, then if provided $\lambda_0\le \frac{1}{12}$, we have $\alpha_1\in (0,1)$. Also, for $\delta_{3s}\le 0.05$, we have $d\le 3.24$.
\end{proof}

\section{Conclusion}\label{section:conclusion}

We have proposed a second-order method named HTP for sparse phase retrieval problem, which is inspired by the hard thresholding pursuit method introduced for compressed sensing. Theoretical analysis illustrates the finite step convergence of the proposed algorithm, which has also been confirmed by numerical experiments. Moreover, numerical experiments also show that our algorithm outperforms the comparative algorithms such as ThWF, SPARTA, CoPRAM significantly in terms of CPU time --- our HTP algorithm can be several times faster than others. Furthermore, there are many other efficient algorithms in compressed sensing that are also convergent in finite steps. It is interesting to investigate such algorithms for sparse phase retrieval.


The results of this paper can be extended to the case when the Gaussian matrix $A$ and underlying signal are complex. However, similar to the mentioned algorithms above, the proposed HTP for sparse phase retrieval is only applicable for standard Gaussian model, due to the fact that the random Gaussian sensing matrix is essential for the success of support recovery at initialization stage. It should be more challenging to design efficient algorithms for the models based on complex Fourier sampling, which shall be the future work.

\section*{Acknowledgment}
The authors would like to thank the anonymous referees for their constructive comments, which have led to an improvement in the quality of this paper. The work of J.-F. Cai is partially supported by Hong Kong Research Grants Council (HKRGC) GRF grants No.~16309219 and No.~16310620. J.-Z. Li is partially supported by the National Science Foundation of China No.~11971221, Guangdong NSF Major Fund No.~2021ZDZX1001, the Shenzhen Sci-Tech Fund No.~RCJC20200714114556020, JCYJ20200109115422828 and JCYJ20190809150413261, and Guangdong Provincial Key Laboratory of Computational Science and Material Design No.~2019B030301001. X.-L. Lu is partially supported by the National Key Research and Development Program of China No.~2020YFA0714200, the National Science Foundation of China No.~11871385 and the Natural Science Foundation of Hubei Province No.~2019CFA007.

\bibliographystyle{abbrv}
\bibliography{phaseretrieval}
\end{document}